\documentclass[11pt,a4paper,reqno]{amsart}
\usepackage{tabls}
\usepackage{url}
\usepackage[linktocpage = true]{hyperref}
\hypersetup{
 colorlinks = true,
 citecolor = blue,
}
\usepackage{color}
\usepackage{relsize}
\definecolor{red}{rgb}{1,0,0}
\definecolor{green}{rgb}{0,1,0}
\definecolor{blue}{rgb}{0,0,1}
\definecolor{refkey}{gray}{.625}
\definecolor{labelkey}{gray}{.625}
\usepackage[lite]{amsrefs}
\usepackage[a4paper]{geometry}
\usepackage{amsfonts}
\usepackage{amssymb}
\usepackage{mathrsfs}
\usepackage{amsmath}
\usepackage{amsthm}
\usepackage{amscd}
\usepackage{enumerate}
\usepackage{paralist}
\usepackage{graphicx}
\usepackage{mathtools}
\usepackage{tikz}
\usepackage{tikz-cd}
\usetikzlibrary{decorations.pathmorphing}

\allowdisplaybreaks

\newcommand{\abs}[1]{\lvert#1\rvert}

\newcommand{\A}{\mathscr{A}}
\newcommand{\B}{\mathcal{B}}

\newcommand{\id}{\operatorname{id}}

\newcommand{\Z}{\mathbb{Z}}

\newcommand{\g}{\mathfrak{g}}
\newcommand{\largeg}{\mathfrak{L}}
\newcommand{\h}{\mathfrak{g}}

\newcommand{\LM}{\mathcal{LM}}

\makeatletter
 \def\title@font{\normalsize\bfseries}
 \let\ltx@maketitle\@maketitle
 \def\@maketitle{\bgroup%
 \let\ltx@title\@title%
 \def\@\title{\resizebox{\textwidth}{!}{%
  \mbox{\title@font\ltx@title}%
 }}%
 \ltx@maketitle%
 \egroup}
\makeatother

\theoremstyle{plain}
\newtheorem{lem}[equation]{Lemma}
\newtheorem{Cor}[equation]{Corollary}
\newtheorem{Thm}[equation]{Theorem}
\newtheorem{theorem}{Theorem}
\newtheorem*{theorem*}{Theorem}
\newtheorem{Def}[equation]{Definition}
\newtheorem{def-prop}[equation]{Definition-Proposition}
\newtheorem{prop}[equation]{Proposition}
\newtheorem{prop-def}[equation]{Proposition-Definition}
\newtheorem{Ex}[equation]{Example}
\newtheorem{Rem}[equation]{Remark}

\numberwithin{equation}{section}
\theoremstyle{remark}

\begin{document}
\def\C{\mathbb{C}}
\def\CE{\mathrm{CE}}
\def\D{\mathcal{D}}
\def\E{\mathscr{E}}
\def\F{\mathscr{F}}
\def\H{\textbf{H}}
\def\k{\mathbb{K}}
\def\G{\mathcal{G}}
\def\M{\mathcal{M}}
\def\m{\mathfrak{m}}
\def\t{\mathfrak{t}}
\def\O{\mathcal{O}}
\def\P{\mathcal{P}}
\def\U{\mathcal{U}}
\def\V{\mathscr{M}}
\def\L{\mathcal{L}}
\def\X{\mathbb{X}}
\def\Y{\mathbb{Y}}
\def\spec{\text{spec}}
\def\Im{\text{Im}}
\def\coker{\operatorname{coker}}
\def\dg{\mathbf{dg}}
\def\Ho{\operatorname{Ho}}
\def\fin{\operatorname{f}}
\def\Ext{\operatorname{Ext}}
\def\End{\operatorname{End}}
\def\pr{\operatorname{pr}}
\def\ad{\operatorname{ad}}
\def\id{\operatorname{id}}
\def\cn{\operatorname{cn}}
\def\Der{\operatorname{Der}}
\def\Hom{\operatorname{Hom}}
\def\Jet{\operatorname{Jet}}
\def\dgLie{\mathbf{dgLie}}
\def\gLie{\mathbf{gLie}}
\def\Cone{\operatorname{Cone}}
\def\dgLeib{\mathbf{dgLeib}}
\def\LM{\mathcal{L}\mathcal{M}}
\def\Map{\operatorname{Map}}
\def\Mod{\mathbf{Mod}}
\def\lfMod{\mathbf{lfMod}}
\def\Rep{\operatorname{Rep}}
\def\Module{\mathbf{Mod}}

\def\RN{\operatorname{RN}}
\def\sgn{\operatorname{sgn}}
\def\sh{\operatorname{sh}}

\newcommand{\tot}{\mathrm{tot}}
\newcommand{\img}{\mathrm{img}}
\newcommand{\mRepginfty}{\mathbf{Mod}_{\g,\infty}}
\newcommand{\Linfty}{\mathcal{L}_{\infty}}
\newcommand{\bb}{\mathfrak{B}}
\newcommand{\TminusoneBg}{T_{[-1]} B\g}
\newcommand{\LP}{Loday-Pirashvili }
\newcommand{\LPM}{\mathbf{LP}}
\newcommand{\LF}{locally finite}
\newcommand{\fgsMod}{\mathbf{fgsMod}}
\newcommand{\firstC}{C}
\newcommand{\secondC}{C}
\newcommand{\dgalgcat}{\mathbf{dgAlg}}

\newcommand{\dadjoint}{d^{\mathrm{ad}}}

\newcommand{\Leib}{\mathbf{Leib}}

\newcommand{\RT}{\mathbf{RT}}

\newcommand{\Oder}{\mathbf{ML}}
\newcommand{\weakLPmodg}{\mathbf{wLP}_{\g}}
\newcommand{\weakLPmodgprime}{\mathbf{wLP}_{\g^\prime}}

\newcommand{\MO}{{monotonic labeling}}
\newcommand{\tohandle}[1]{\textcolor{red}{#1}}
\newcommand{\doubleabs}[1]{\lVert  #1  \rVert}
\newcommand{\tanH}{H_{\mathrm{tan}}}

\title[Locally finite infinity-modules and weak Loday-Pirashvili modules]{Locally finite infinity-modules and weak Loday-Pirashvili modules over differential graded Lie algebras}

\author{Zhuo Chen}
\address{Department of Mathematics, Tsinghua University, Beijing 100084, China}
\email{\href{chenzhuo@mail.tsinghua.edu.cn}{chenzhuo@tsinghua.edu.cn}}

\author{Yu Qiao*}
\address{(Corresponding author) School of Mathematics and Statistics, Shaanxi Normal University, Xi'an 710119, China}
\email{\href{yqiao@snnu.edu.cn}{yqiao@snnu.edu.cn}}

\author{Maosong Xiang}
\address{School of Mathematics and Statistics, Center for Mathematical Sciences, Huazhong University of Science and Technology, Wuhan 430074, China}
\email{\href{mailto: msxiang@hust.edu.cn}{msxiang@hust.edu.cn}}

\author{Tao Zhang}
\address{College of Mathematics and Information Science, Henan Normal University, Xinxiang 453007, China}
\email{\href{mailto:zhangtao@htu.edu.cn}{zhangtao@htu.edu.cn}}

\thanks{Research partially supported by NSFC grants 12071241 (Chen), 11971282 (Qiao), 11901221 (Xiang), and 11961049 (Zhang).}

\begin{abstract}
Motivated by recent developments of $\infty$-categorical theories  related to differential graded (dg for short) Lie algebras,   we develop a general framework for locally finite $\infty$-$\g$-modules over a dg Lie algebra $\g$.  We show that the category of such locally finite $\infty$-$\g$-modules   is   almost a model category in the sense of Vallette. As a homotopy theoretical generalization of Loday and Pirashvili's  Lie algebra objects in the tensor category of linear maps, we further study weak {\LP modules}  consisting of {$\infty$-morphisms} from {\LF} $\infty$-$\g$-modules to the adjoint module   $\g$. From the category of such weak \LP modules over $\g$, we find a functor that maps to the category   of Leibniz$_\infty$ algebras enriched over the Chevalley-Eilenberg dg algebra of $\g$. This functor can be regarded as the homotopy lifting of Loday and Pirashvili's original method to realize Leibniz algebras from Lie algebra objects in the category  of linear maps.
\end{abstract}

\maketitle

\subjclass{{\em Mathematics Subject Classification} (2020): Primary 17B55; Secondary 17A32, 18G35, 18N40.}
\\
\keywords{{\em Keywords}: dg Lie algebra, locally finite $\infty$-module, weak Loday-Pirashvili module, model category, Leibniz$_\infty$ algebra.}

\setcounter{tocdepth}{3}
\setcounter{secnumdepth}{3}

\tableofcontents

\section*{Introduction}
The subject for investigation in this paper is concerned with   differential graded (abbreviated as ``dg'') Lie algebras and a generalized type of their modules. We first briefly review the background  and significance of dg Lie algebras. In fact, they appear widely in deformation theory due to the  principle by Gerstenhaber, Quillen, Deligne, Drinfeld, Feigin, Kontsevich and others,  that  every deformation problem (in characteristic $0$) is controlled by a dg Lie algebra via solutions of Maurer-Cartan equation modulo gauge equivalence  and  with quasi-isomorphic dg Lie algebras giving the same deformation problem (\cites{Gerstenhaber1964,Quillen1969,Goldman-Millson,Deligne}). This general heuristic way of thinking has powerfully influenced the works of Hinich \cite{Hinich}, Kontsevich-Soibelman \cites{KS-1,KS-2}, Manetti \cites{Manetti-1,Manetti-2,Manetti-3}, and many others.
For a precise formulation of this fact in the language of higher category theory, see a nice synopsis from Lurie's 2010 ICM address  \cites{DAGX} or \cites{DAGX2}. See also Pridham's independent work \cite{Pridham}. For  more about the derived deformation theory (DDT) program studying formal moduli problems with the central concept of  Koszul duality between commutative algebras and dg Lie algebras, see the references \cites{BarannikovK1998,Fontanine-Kapranov,CalaqueGrivaux}.
Recent progresses in \cites{BMnew,BR} shed light on   applications of finite dimensional dg Lie algebras and their representations.

Motivated by these remarkable works and ideas,  our research focuses on dg Lie algebras and their representations, and is particularly inspired by the papers \cites{RUTHAbdCrainic,Mehta}, where it is explained, based on Quillen's superconnections \cite{Quillen1985}, what   the ``right'' definition of a Lie algebra representation up to homotopy is, and by \cite{PaviaMixed}, where it is studied in depth the $\infty$-categorical construction of mixed graded structures
(in the sense of \cites{PTVV13, CalaqueGrivaux, CalaquePantevToenVaquieVezzosiJT}) over Chevalley-Eilenberg complexes.
We think it necessary to specialize in  generalized modules of dg Lie algebras, and this article is devoted to locally finite $\infty$-modules --- Let $\g$ be a dg Lie algebra and $d_{\CE}$ the differential on the Chevalley-Eilenberg algebra $C(\g)=\sum_{i=0}^\infty C^i(\g):=\sum_{i=0}^\infty S^i(g^\vee[-1])$.
By saying a locally finite $\infty$-$\g$-module, we mean a graded space $M$ together with a family of degree $(+1)$ linear maps $d^M_i:$ $M\to C^i(\g,M):=C^i(g )\otimes M $ such that $d_{\tot}^M :=d_{\CE} +\sum_{i=0}^\infty d_i^M$, called the \textbf{total differential}, defines a square zero derivation on $C(\g,M):=C(\g)\otimes M$, and for each $m\in M$, there are only finitely many $d_i^M(m)$ which are nontrivial (locally finiteness). For details, see Definition \ref{Def:lfinftygmodule}. The local finiteness condition ensures that $d_{\tot}^M$ is well-defined.
See~\cite{BM} for a similar condition.

Locally finite $\infty$-modules naturally arise when we study the realization of modules over a dg Lie algebra.
Consider the Chevalley-Eilenberg functor $C(\g,-)$ from the category $\Module_{\g}$ of (ordinary) $\g$-modules to the category $\Mod_{C(\g)}$ of dg $C(\g)$-modules.
Although this functor is faithful, it is not full. To get a fully faithful extension, one has to enlarge the category $\Module_{\g}$ to a bigger one ${\mRepginfty}$ by adding $\infty$-morphisms  (also known as weak morphisms)  of $\g$-modules with a  local finiteness constraint. In doing so, we obtain a \textit{minimal fully faithful extension}  of $C(\g,-)$.
The category $\mRepginfty$ is still insufficient for our purpose because  pullbacks of (trivial) fibrations in this category do not exist in general.
We find that just adding {\LF} $\infty$-$\g$-modules would meet our requirements ---  First, define  the category $\lfMod_\g^\infty$ of {\LF} $\infty$-$\g$-modules whose arrows are {$\infty$-morphisms} (see Definition \ref{Def-Prop:inftymorphism}). Roughly speaking,
an  \textbf{$\infty$-morphism}  of {\LF} $\infty$-$\g$-modules $\phi\colon  M \rightsquigarrow  {M^\prime}$ is a set of linear maps
\begin{equation}\label{Eqt:phiphik}
	\phi = \{\phi_k   \colon M\to C^{ k} (\g,M^\prime) \}_{k\geqslant 0}
\end{equation}
subject to some natural constraint (intertwining the relevant total differentials). Second,  we show that the classical Chevalley-Eilenberg functor $C(\g,-)$ admits a fully faithful extension, again denoted by $\secondC(\g,-)$, from $\lfMod_\g^\infty$ to $\Mod_{C(\g)}$, which gives rise to an equivalence between the category $\lfMod_\g^{\infty}$ of {\LF} $\infty$-$\g$-modules and the subcategory $\fgsMod_{C(\g)}$ of semifree dg $C(\g)$-modules (see Proposition~\ref{prop: CE functor from infty modules}).

There is a natural definition of weak equivalence in the category $\lfMod_\g^\infty$:
an \textbf{$\infty$-morphism} as in \eqref{Eqt:phiphik} is called a \textbf{weak equivalence} if the cochain map $\phi_0\colon (M, d^M_{0}) \to (M^\prime, d^{M^\prime}_{0})$ is a quasi-isomorphism of cochain complexes. Weakly equivalent objects in such situations
are regarded as ``being the same'',  in view of the sophisticated concepts in model category theory \cites{DwyerSpalinski,Hirschhorn,Hovey}, which  is a powerful framework   describing
categories with weak equivalences together with categorical constructions that are
compatible with the weak equivalences.

The aims of this paper are two-folded: to show that $\lfMod_\g^\infty$ admits \textit{almost a model category structure\footnote{However, the category $\lfMod_\g^\infty$ does not admit a model category structure, since it does not necessarily have all small limits and colimits.}} in the sense of Vallette~\cite{Vallette}  and to illustrate the immense relevance of such structure both from a conceptual perspective and from a more practical point of view. Here is our first main theorem.
\begin{theorem}[Theorem~\ref{Thm: Repinftyg is fibrant}]
	The category $\lfMod_\g^{\infty}$ of {\LF} $\infty$-modules over a dg Lie algebra $\g$ is  a category of fibrant objects   in the sense of Brown~\cite{Brown}, in which a  morphism $\phi$ as in \eqref{Eqt:phiphik}
is defined to be  a fibration if   the cochain map $\phi_0 \colon (M,d^M_{0}) \to (M^\prime, d^{M^\prime}_{0})$ is degreewise surjective.
	
We call $\phi$   a cofibration if   the cochain map $\phi_0 \colon (M,d^M_{0}) \to (M^\prime, d^{M^\prime}_{0})$ is degreewise injective. Then the category $\lfMod_\g^{\infty}$ is almost a model category in the sense of Vallette~\cite{Vallette}, and each object in $\lfMod_\g^{\infty}$ is both fibrant and cofibrant\footnote{We believe that this result is known to experts but we could not find a complete statement and proof in the existing literature.}.
\end{theorem}

Another important source of inspiration is the close   connection between Leibniz algebras and  Lie algebras revealed by the works  of   Loday and Pirashvili \cites{LPu,LP}. Here is a quick sketch ---
The notion of Leibniz algebras (also known as Loday algebras)   is  a non-commutative generalization of Lie algebras (cf. \cite{Loday}). In \cite{LPu}, Loday and Pirashvili constructed and studied the universal enveloping algebra functor $UL$ from the category of Leibniz algebras to that of associative algebras.
Unlike the universal enveloping algebra functor $U$ from the category of Lie algebras to that of associative algebras, the functor $UL$ does not factor through the category of cocommutative Hopf algebras.
In order to understand the factorization nature of Leibniz algebras, Loday and Pirashvili introduced in~\cite{LP} the tensor category $\mathcal{LM}$ of linear maps. They proved that there exists an adjunction between the category of Leibniz algebras and the category of Lie algebra objects in $\mathcal{LM}$, such that the functor $UL$ factors through the category of Hopf objects in $\mathcal{LM}$.
Moreover, unravelling a  Lie algebra object in $\mathcal{LM}$ yields a morphism  from   a Lie algebra module to the Lie algebra itself viewed as the adjoint module. In this article and in  our previous work \cite{Qiaoold}, we refer to such a  morphism as  a \textit{\LP module}.
The adjunction between the category of Leibniz algebras and that of Lie algebra objects in $\mathcal{LM}$, can be enhanced to the dg setting. That is, there exists an adjunction between the category of dg Leibniz algebras and the category of morphisms of modules of a dg Lie algebra with target being the adjoint module (see Section \ref{Sec:adjunction}).
We call the latter one the category of   {\LP modules} (over a given dg Lie algebra), and denote it by ${\LPM }$. Objects in ${{\LPM}}$ can be viewed as adjoint module points in the category of ordinary dg Lie algebra modules in the sense of functor of points due to Grothendieck.

Note that the adjoint module $\g$ is a particular locally finite $\infty$-$\g$-module.
Inspired by Loday and Pirashvili's foundational minds, as reviewed earlier, we focus on  ``adjoint-points" in the category $\lfMod_\g^\infty$. Such a ``point'' is just an   {$\infty$-morphism} $f\colon M \rightsquigarrow \g$ from a {\LF} $\infty$-$\g$-module $M$ to the adjoint module $\g$, and we call it  a \textit{weak \LP module}. To expand, $f$   consists of a family of maps $\{f_k \colon M \to C^k(\g; \g)\}_{k\geqslant 0}$ satisfying certain conditions (see Definition~\ref{Def: wLP}).
The collection of such {$\infty$-morphisms} forms a category, which we call the category of \textit{weak {\LP modules}} and denote by ${\weakLPmodg}$.
Two weak \LP modules over $\g$ are said to be homotopic if they are homotopic as $\infty$-morphisms of {\LF} $\infty$-$\g$-modules.

It is worth pointing out that weak \LP modules  have a geometric characterization in formal dg geometry. In fact, each dg Lie algebra $\g$ can be viewed as a (formal pointed) dg manifold $B\g = (\g[1],d_{\CE})$~\cite{MSX}, whose space of functions $C^\infty(B\g)$ is naturally identified with the commutative dg algebra $C(\g)$. Moreover, we have the following facts.
\begin{enumerate}
	\item The category $\fgsMod_{C(\g)}$ of finitely generated and semifree  $C(\g)$-modules is equivalent to the category $\mathbf{dgVec}_{B\g}$ of dg vector bundles over $B\g$;
	\item The semifree $C(\g)$-modules $C(\g,\g)$ is identified with the space of global sections of the shifted tangent bundle $\TminusoneBg $ of the dg manifold $B\g$; and
	\item The category ${\weakLPmodg}$ is equivalent to the category $\left(\mathbf{dgVec}_{B\g} \right)_{ C(\g,\g) }$ of dg vector bundles of $B\g$ over its shifted tangent bundle $\TminusoneBg $ (see Proposition~\ref{prop-def via dg bundle}).
\end{enumerate}

As the second main result of this paper,  we find that Leibniz$_\infty$ algebras show  up naturally from the category $\weakLPmodg$ of weak \LP modules over $\g$.    Here the notion of  Leibniz$_\infty$ algebras (also known as homotopy Leibniz   or Loday infinity algebras) is a homotopy replacement  of that of ordinary Leibniz algebras.
The category of Leibniz$_\infty$ algebras  enjoys some nice phenomena as its subcategory of $\Linfty$  algebras first introduced by Lada and Stasheff ~\cite{LS} (also known as strongly homotopy Lie algebras), e.g., the minimal model theorem  proved by Ammar and Poncin~\cite{AP}. However, unlike $\Linfty$ algebras which abound in various formal deformation theories, the method that produces Leibniz$_\infty$ algebras but not $\Linfty$ algebras  is relatively rare in existing research  (see~\cite{Uchino} for a construction from derived bracket).
Our point  is to show an interesting mechanism of how to obtain Leibniz$_\infty$ algebras ---  Arising from any weak {\LP module} $f \colon M \rightsquigarrow \g$ we find a canonical Leibniz$_\infty$ algebra $C(\g,M)$ which is in general not an $\Linfty$ algebra, and its unary bracket is the total differential $d^M_{\tot}$. The details are given in Section \ref{Section: HLA}, and
the process of constructing  Leibniz$_\infty$ algebras constitutes a functor with good properties.
\begin{theorem}[Theorem~\ref{Thm:wLP to HLA}]\label{Thm B}
	Given a dg Lie algebra $\g$, there exists a functor
	\[
	\mathcal{G}_\g^\infty \colon  \weakLPmodg \to {\Leib^\infty_{C(\g)}}
	\]
	from the category $\weakLPmodg$ of weak \LP modules to the category ${\Leib^\infty_{C(\g)}}$ of Leibniz$_\infty$ algebras over the commutative dg algebra $C(\g)$ which is  homotopy invariant (i.e., $\mathcal{G}_\g^\infty$ sends homotopic weak \LP modules to the same Leibniz$_\infty$ algebra over $C(\g)$).
\end{theorem}
The functor $\mathcal{G}^\infty_\g$  can be viewed as a homotopy lifting of Loday and Pirashvili's original method to realize Leibniz algebras from Lie algebra objects in the category $\mathcal{LM}$~\cite{LP}.
Let us provide a brief  explanation  --- Given a weak \LP module $f\colon M \rightsquigarrow \g$,  its tangent cohomology $\tanH(f_0) \colon \tanH^\bullet(M) \to \tanH^\bullet(\g)$ defines a \LP module, i.e., a graded Lie algebra morphism from the $\tanH^\bullet(\g)$-module $\tanH^\bullet(M)$ to $\tanH^\bullet(\g)$. In turn, we obtain a Leibniz algebra structure on $\tanH^\bullet(M)$, which is indeed compatible   with $\mathcal{G}^\infty_\g(f\colon M \rightsquigarrow \g)$, i.e., the Leibniz$_\infty$ algebra   $C(\g,M)$  (see Proposition \ref{main prop}).

Moreover, $\mathcal{G}^\infty_\g$ is also functorial with respect to variations of dg Lie algebras.
In fact, the collection of all weak Loday-Pirashvili modules forms a category, denoted by $\mathbf{wLP}$. Our construction in Theorem~\ref{Thm B} indeed gives rise to a functor $\mathcal{G}^\infty$ from $\mathbf{wLP}$ to the category $\mathbf{Leib}^\infty$ of Leibniz$_\infty$ algebras.
However, it is still mysterious whether the functor $\mathcal{G}^\infty$ admits a left adjoint satisfying the condition that the associated adjunction is a homotopy lifting of the adjunction between Leibniz algebras and Loday-Pirashvili modules  (see Proposition~\ref{prop: functorial wrt g} and Remark \ref{Rem:furtherthoughts}). We leave this problem to the future.

To describe structure maps of the Leibniz$_\infty$ $C(\g)$-algebra   $\secondC(\g,M)=\mathcal{G}^\infty_\g(f\colon M \rightsquigarrow \g)$  in a more vivid way, we need rooted forests with {\MO}s. The idea comes from  rooted trees (with {\MO}s) which are utilized to formulate   the generalized Campbell-Hausdorff series~\cite{Getzler}*{Proposition 5.9}. Our Proposition \ref{Prop: Leibniz infty via OPT} gives an interesting formula, the step by step process  giving all structure maps of $\secondC(\g,M)$  through a tree traversal algorithm.

Finally, we  point out that if $\g$ in this article is assumed to be  a locally finite  $\Linfty$  algebra, a natural generalization of  dg Lie algebras  \cite{LS},   then almost all of our results  (with some necessary adaptations) will in fact hold.
 Further, we would   like to  mention  a non-linear analogue of the close connection between Leibniz algebras and dg Lie algebras --- the correspondence between augmented racks and cubical monoids, see \cites{Clauwens, Mostovoy}. See also a recent work \cite{Mostovoy2}.

\textbf{Structure of the paper.}
In Section \ref{Sec:lfinftymodules}, we collect some basic definitions and results about dg Lie algebras, their modules, and the category of locally finite $\infty$-modules over a dg Lie algebra.  We dedicate  Section \ref{Sec:Proof} to giving the full details of our main Theorem \textbf{A}/\ref{Thm: Repinftyg is fibrant}.
In Section \ref{Sec:LPmodule},  we  introduce and exploit  weak {\LP modules},  and show a list of important facts and identities that are subsequently used in Sections \ref{Section: HLA}  to establish   Theorem \textbf{B}/\ref{Thm:wLP to HLA}.   The last Section \ref{Sec:rootedtree}  illustrates how the   Leibniz$_\infty$ algebras we constructed are realized via rooted forests with {\MO}s.

\textbf{Terminology and notations:}

\textit{Ground field, dimensions, and degrees.}
We use the symbol $\k$ to denote the field of real or complex numbers. All vector spaces are  \textbf{finite dimensional} over $\k$ unless otherwise stated. In particular, all dg Lie algebras that we consider are finite dimensional as well. By ``graded" we mean $\Z$-graded. The degree of a homogenous element $m$ in a graded vector space $M =\oplus_{i\in\Z} M^{i}$ is denoted by $\abs{m}$.

\textit{Shuffles.}
A $(p, q)$-shuffle is a permutation $\sigma$ of the set $\{1, 2, \cdots, p+q\}$ such that
$\sigma(1)<\cdots < \sigma(p)$ and $\sigma(p+1)<\cdots < \sigma(p+q)$.
The symbol $\mathrm{Sh}(p, \, q)$ denotes the set of $(p, \, q)$-shuffles.

\textit{Symmetric algebra.} Let $M$ be a graded vector space. We use $S(M):=\oplus_{p \geqslant 0}S^p(M)$ to denote the compactly supported symmetric algebra generated by $M$.
The symbol $\widehat{S}(M)$ denotes the $\mathfrak{m}$-adic completion of   $S(M)$,
where $\mathfrak{m}$ is the maximal ideal of $S(M)$ generated by $M$. In other words, we can treat $\widehat{S}(M)=\prod_{p \geqslant 0}S^p(M)$.

\textit{Koszul sign.}
The Koszul sign $\kappa(\sigma; \, m_1, \,  \cdots, \, m_n)$ of a permutation $\sigma$ of
homogeneous elements $m_1, \cdots, m_n$ in a graded vector space $M$ is determined by the relation
\begin{equation*}
m_1 \odot \cdots \odot  m_n=\kappa(\sigma; \, m_1, \, \cdots, \, m_n) m_{\sigma(1)} \odot \cdots \odot  m_{\sigma(n)}
\end{equation*}
in the graded commutative algebra $S(M)$. Unless specified otherwise, we use the abbreviated notation $\kappa(\sigma)$.

\textit{Dg algebra modules.}
Suppose that $\A=(\A,d_\A)$ is a commutative dg algebra. By saying an $\A$-module we mean a pair $(D,\partial^D)$, where $D$ is a (graded) module over the (graded) commutative algebra $\A$ and  $\partial^D:~D\to D[1]$ is a differential which satisfies
\[
\partial^D(am)=(d_\A a) m+(-1)^{a}a \partial^D(m),\quad\forall~ a\in \A,m\in D.
\]
What is referred to herein as an $\A$-module is called a dg $\A$-module in \cites{CLX,Qiaoold}.

\textit{Some commonly used notations.}
\begin{compactenum}
\item $\dgLie$ --- the category of dg Lie algebras; Objects are denoted by $\g, \mathfrak{h}$, etc..
\item $C(\g)$ --- the Chevalley-Eilenberg dg algebra of a dg Lie algebra $\g$;
\item ${\mathbf{Ch}}(\k)$ --- the category of cochain complexes of vector spaces;
\item $\Module_{\g}$ --- the category of $\g$-modules for a dg Lie algebra $\g$;
\item $\Mod_{C(\g)}$ --- the category of   $C(\g)$-modules;
\item   ${\LPM }$ --- the category of  \LP modules;
\item ${\LPM _{\g}}$ --- the subcategory of ${\LPM }$ consisting of  {\LP modules}  over $\g$;
\item $\dgLeib $ --- the category of dg Leibniz  $\k$-algebras;
\item ${\dgLeib}_{\A}$   --- the category of dg Leibniz algebras over  a dg algebra $\A$;
\item $\lfMod_\g^\infty$ --- the category  of  {\LF} $\infty$-$\g$-modules;	
\item  $\fgsMod_{C(\g)} $ --- the category of finitely generated and semifree   $C(\g)$-modules;
\item ${\weakLPmodg}$ --- the category  of weak {\LP modules} over $\g$;
\item $ \mathbf{wLP} $ --- the category of weak {\LP modules};
\item $\Ho(\lfMod_\g^\infty)$ --- the homotopy category of {\LF} $\infty$-$\g$-modules;
\item $\Ho(\weakLPmodg)$ --- the category of  homotopy classes of weak {\LP modules} over $\g$;
\item $\Leib^{\infty[1]}_{C(\g)}$ --- the category of  Leibniz$_\infty[1]$ algebras over the commutative dg algebra $C(\g)$;
\end{compactenum}

\textbf{Acknowledgment.}
We would like to thank Jiahao Cheng, Zhangju Liu, Feng Qu, and Guodong Zhou for fruitful discussions.

\section{Locally finite infinity-modules over dg Lie algebras}\label{Sec:lfinftymodules}
\subsection{Dg Lie algebras and Chevalley-Eilenberg functors}
 We start with some basic knowledge of dg Lie algebras.
\begin{Def}[\cite{DAGX}]\label{Def:dgLiealgebra}
	A \textbf{dg Lie algebra}  over $\k$ is a triple $\g=(\g,d,[-,-])$  where $(\g,d)$ forms a cochain complex of vector spaces,  and $[-,-]\colon \g  \otimes \g  \to \g $ is a degree $0$ bilinear map  called the graded Lie bracket, satisfying the following conditions:
	\begin{compactenum}
		\item For all $x \in \g^p, y \in \g^q$, we have $[x,y] = -(-1)^{pq}[y,x]$;
		\item For all $x \in \g^p, y \in \g^q, z \in \g^r$, we have
		\[
		(-1)^{pr}[x,[y,z]] + (-1)^{qp}[y,[z,x]] + (-1)^{rq}[z,[x,y]] = 0;
		\]
		\item The differential $d$ is a derivation with respect to the graded Lie bracket $[-,-]$, i.e., for all $x \in \g^p, y \in \g^q$, we have
		\[
		d[x,y] = [dx,y] + (-1)^p[x,dy].
		\]
	\end{compactenum}
	A morphism of dg Lie algebras from $\g = (\g,d,[-,-])$ to $\g^\prime = (\g^{\prime }, d^\prime, [-,-]^\prime)$ is a map of cochain complexes $f\colon (\g, d) \to (\g^{\prime },d^\prime)$ such that $f([x,y]) = [f(x),f(y)]^{\prime}$ for all $x,y \in \g$.
	The collection of all dg Lie algebras and their morphisms forms a category, denoted by $\dgLie$.
\end{Def}
Given a dg Lie algebra $\g$, the cohomology of the underlying cochain complex $(\g,d)$, called the \textbf{tangent cohomology} and denote by $\tanH^\bullet(\g)$, when equipped with the induced bracket $[-,-]$, is a graded Lie algebra. Each morphism  $f \colon \g \to \g^\prime$ of dg Lie algebras induces a morphism of graded Lie algebras   $\tanH(f) \colon (\tanH^\bullet(\g), [-,-]) \to (\tanH^\bullet(\g^\prime), [-,-]^\prime)$. Hence, $\tanH$ is indeed a functor from $\dgLie$ to $\gLie$, the category of  graded  Lie algebras.

The \textbf{Chevalley-Eilenberg dg algebra} $C(\g)$ of a dg Lie algebra $\g$ is defined as follows:
\begin{compactenum}
	\item The underlying space is the \textbf{compactly supported} symmetric algebra of $\g^{\vee}[-1]$, i.e.,
	\[
	C(\g) := S(\g^{\vee}[-1]) = \oplus_{p \geqslant 0} S^p(\g^{\vee}[-1]).
	\]
	Here ``compactly supported" means that we take \textit{direct sum}, instead of direct product. The number $p$ is referred to as the \textbf{weight} of elements in $C^p(\g):=S^p(\g^{\vee}[-1])$.
	\item The differential $d_{\CE}$ on $C(\g)$, called the \textbf{Chevalley-Eilenberg differential}, is the sum  $d_0 + d_1$, where
	\[
	d_0 \colon C^\bullet(\g) \to C^\bullet(\g)[1],
	\]
	is induced from the dual of the differential $d \colon \g \to \g[1]$ of the complex $(\g, d)$, and
	\[
	d_1 \colon C^\bullet(\g) \to C^{\bullet+1}(\g)[1]
	\]
	is induced from the graded Lie bracket $[-,-]$ by
	\[
	d_1(\xi[-1])(x[1], y[1]) = (-1)^{\abs{\xi}+1+\abs{x}}\xi([x,y]),\quad \forall~ \xi \in \g^\vee, x, y \in \g.
	\]
\end{compactenum}
Note that a morphism of dg Lie algebras $f \colon \g \to \g^\prime$ corresponds to a degree $0$ and weight-preserving morphism of dg algebras $f^\vee \colon (C(\g^\prime), d^\prime_{\CE}) \to (C(\g),d_{\CE})$.
Thus, we obtain a functor
\[
C \colon \dgLie \to {\dgalgcat}, \quad \g \mapsto C(\g),
\]
called the Chevalley-Eilenberg functor, from the category $\dgLie$ of dg Lie algebras  to the category ${\dgalgcat}$ of dg algebras.
\begin{Def}[\cite{DAGX}]\label{Def:dgmodule}
	Let $\g = (\g, d, [-,-])$ be a dg Lie algebra. A \textbf{representation} of $\g$, or a \textbf{$\g$-module}, is a finite dimensional cochain complex $M = (M,d^M)$ of vector spaces equipped with a bilinear map $\rho_{\g,M} \colon \g \otimes M \to M$ satisfying the following conditions:
	\begin{compactenum}
		\item $\rho_{\g,M}$ is a morphism of cochain complexes, i.e.,
		\begin{align*}
		d^M (\rho_{\g,M}(x; m)) &= \rho_{\g,M}(dx; m) + (-1)^{\abs{x}} \rho_{\g,M}(x;d^M m ),
		\end{align*}
		for all homogeneous $x ,y \in \g$ and $m\in M$;
		\item $\rho_{\g,M}$ is compatible with the graded Lie bracket $[-,-]$ on $\g$, i.e.,
		\begin{align*}
		\rho_{\g,M}([x,y]; m) &= \rho_{\g,M}(x;\rho_{\g,M}(y; m)) - (-1)^{\abs{x}\abs{y}}\rho_{\g,M}(y; \rho_{\g,M}(x; m)),
		\end{align*}
		for all homogeneous $x ,y \in \g$ and $m\in M$.
	\end{compactenum}
	Given two $\g$-modules $M$ and ${N}$, a morphism from $M$ to ${N}$ is a morphism $\phi\colon M=(M,d^M) \to N=(N, d^N)$ of the underlying cochain complexes such that the following diagram
	\[
	\begin{tikzcd}
	\g \otimes M \ar{d}[left]{\id \otimes \phi} \ar{r}{\rho_{\g,M}} & M \ar{d}{\phi} \\
	\g \otimes N \ar{r}{\rho_{\g,W}} & W
	\end{tikzcd}
	\]
	commutes in the category ${\mathbf{Ch}}(\k)$ of cochain complexes of vector spaces.
	
	The modules of a dg Lie algebra $\g$ and their morphisms form a category which we denote by $\Module_{\g}$.
\end{Def}
Given a dg Lie algebra $\g = (\g,d,[-,-])$, the underlying cochain complex $(\g,d)$ admits a $\g$-action defined by the graded Lie bracket $[-,-]$, called the \textbf{adjoint module} of $\g$.

There is a correspondence between dg Lie algebra modules and dg modules over their Chevalley-Eilenberg dg algebras. More precisely, given a dg Lie algebra $\g$, let $C(\g)$ be its Chevalley-Eilenberg dg algebra.
A $\g$-module structure on a cochain complex $M=(M,d^M)$ is equivalent to a  $C(\g)$-module structure on the space
\[
C(\g, M) :=  \bigoplus_{p \geqslant 0}C^p(\g, M) := \bigoplus_{p \geqslant 0} S^p(\g^{\vee}[-1]) \otimes M,
\]
whose differential is of the form
\[
d_{\tot}^M = d_{\CE} +  d_0^M + d_1^M:=  d_{\CE} \otimes {\id_{M}} + \id_{C(\g)} \otimes d^M + \id_{C(\g)}\otimes d_1^M \colon C(\g,M)\to C(\g,M)[1],
\]
where $d_0^M$ is simply extended from $d^M$ (in the sequel we abuse the notations $d^M$ and $d^M_0$) and the weight $1$ component $d_1^M \colon M \to C^1(\g, M)[1]$ corresponds to the $\g$-action map $\rho_{\g,M}$ on $M$ via the formula
\[
\iota_{x[1]} d_1^M(m) = \rho_{\g,M}(x; m),\quad x \in \g, m \in M.
\]
Here $\iota_{x[1]} \colon C^\bullet(\g)\to C^{\bullet-1}(\g)$ is the contraction map of degree $(\abs{x}-1)$. We   call $d_{\tot}^M$ the \textbf{total Chevalley-Eilenberg differential} (or simply total differential).

The  $C(\g)$-module  $C(\g,M)=(C(\g,M), d_{\tot}^M)$ is called the \textbf{total Chevalley-Eilenberg dg module} of the $\g$-module $M$.
Meanwhile, a morphism of $\g$-modules $\phi\colon M\to N$ corresponds to a weight-preserving morphism of $ C(\g)$-modules $C(\g,\phi)\colon C(\g,M)\to  C(\g,N)$, i.e., $C(\g,\phi)$ maps $C^p (\g, M)$ to $ C^{p} (\g, N)$ for all $p\geqslant 0$.
Thus, we obtain a functor
\[
C(\g, -)\colon \Module_{\g} \to \Mod_{C(\g) },\quad M\mapsto C(\g,M),
\]
also called the Chevalley-Eilenberg functor.
Here $\Mod_{C(\g)}$ denotes the category of  $C(\g)$-modules,  whose morphisms are $C(\g)$-linear maps compatible with relevant total differentials.

Viewing $\g$ as the adjoint $\g$-module, we have the total  differential on $C(\g,\g)=C(\g)\otimes\g$ defined by
\begin{equation}\label{Eqt:dtotg}
d_{\tot}^{\g}:=d_{\CE}\otimes {\id_{\g}} +   {\id_{C(\g)}}\otimes d+  {\id_{C(\g)}}\otimes \dadjoint, 	
\end{equation}
where $\dadjoint\colon \g\to \g^{\vee}[-1]\otimes \g [1]$ comes from the adjoint action of $\g$ on $\g$.

The Chevalley-Eilenberg functor $C(\g,-)$ is faithful but \emph{not full}. In fact, given two objects $M$ and $N$ in $\Module_{\g}$, the functor $C(\g,-)$ defines a map
\[
C(\g,-) \colon \Module_{\g}(M,N) \hookrightarrow \Mod_{C(\g)}(C(\g,M), C(\g,N)).
\]
The set of morphisms $\Mod_{C(\g)}(C(\g,M), C(\g,N))$ consists of degree $0$ and $C(\g)$-linear maps $f$ satisfying $d_{\tot}^N \circ f = f \circ d_{\tot}^M$.
The map $f$ is of the form $f = \sum_{p= 0}^{\infty} f_p$,
where  $f_p \colon  C^\bullet(\g,M)  \to C^{\bullet+p}(\g,N)$ is called the \textbf{weight} $p$ component of $f$, whereas the image of $C(\g,-)$ consists only of weight $0$ morphisms $f_0$ by definition. Thus the functor $C(\g,-)$ is not surjective in general.

\subsection{Locally finite infinity-modules}\label{Sec:LFinftygmodule}
Let $\g$ be a dg Lie algebra and $C(\g)$ its Chevalley-Eilenberg dg algebra.
We now introduce the notion of locally finite $\infty$-$\g$-modules, which can be thought of compactly supported $\Linfty$ modules over $\g$.

\begin{Def}\label{Def:lfinftygmodule}
	A \textbf{{locally finite} $\infty$-$\g$-module} is a cochain complex $M=(M, d^M)$ together with a square zero linear operator
	\begin{equation}\label{Eqt:dCEVoriginalinfinity}
	d_{\tot}^M :=d_{\CE}\otimes {\id_{M}}+\sum_{k=0}^\infty \id_{C(\g)}\otimes d_k^M  \colon C(\g,M)\to C(\g,M)[1],
	\end{equation}	
	called the \textbf{total Chevalley-Eilenberg differential} (or simply total differential), satisfying the \textbf{local finiteness constraint}: For all $m \in M$, there exists a sufficiently large integer $N_m$ depending on $m$ such that $d^M_k(m) = 0$ for all $k \geqslant N_m$.
	Here $d_k^M \colon M \to C^{k}(\g,M)$, for all $k\geqslant 0$, is called the weight $k$ component structure map and the initial one $d_0^M$   $:=d^M$.
\end{Def}
In the sequel, a locally finite $\infty$-$\g$-module would be simply denoted by    $(M, d_{\tot}^M)$, or just $M$ whose structure maps are $d_k^M$  by default.
Given a locally finite $\infty$-$\g$-module $(M, d_{\tot}^M)$,  we call $C(\g,M)=(C(\g,M),d_{\tot}^M)$ the \textbf{total Chevalley-Eilenberg  $C(\g)$-module} of $M$. The corresponding cohomology, denoted by $H_{\tot}(\g,M)$, is called the \textbf{total Chevalley-Eilenberg cohomology} of $M$.
Meanwhile, it is easy to see that the map $d_1^M \colon M \to C^1(\g,M)$ induces a module structure of the tangent Lie algebra $\tanH^\bullet(\g)$ on the  cohomology $\tanH^\bullet(M)$ of the cochain complex $(M, d^M_0)$.
Still, we call $\tanH^\bullet(M)$ the \textbf{tangent cohomology} of  the {\LF} $\infty$-$\g$-module $M$.

\begin{Def}\label{Def-Prop:inftymorphism}
	An  \textbf{$\infty$-morphism}  of {\LF} $\infty$-$\g$-modules $f\colon  M \rightsquigarrow  {N}$
	is a family of degree $0$ linear maps $f_k \colon M\to C^{ k} (\g,N) $, called the weight $k$ component of $f$, such that $f=\sum_{k=0}^\infty f_k \colon M\to C(\g,{N})$ is locally finite,  that is, for each homogeneous element $m \in M$, $f_k(m)$ vanishes if $k$ is sufficiently large, and $f$ fits into the following commutative diagram
	\[
	\begin{tikzcd}
	M \ar{d}[left]{d_{\tot}^{M }} \ar{rr}{f} & & C(\g,{N}) \ar{d}{d_{\tot}^{N}} \\
	C(\g,M )  \ar{rr}{F := \id_{C(\g)}\otimes f} && C(\g,{N}).
	\end{tikzcd}
	\]
	Denote by $\lfMod_\g^{\infty}$ the category of {\LF} $\infty$-$\g$-modules and their  $\infty$-morphisms.
\end{Def}

Clearly, the component $f_0\colon M\to N$ of $f$ induces a linear map $\tanH(f_0)\colon \tanH(M)\to \tanH(N)$ which is indeed a morphism of $\tanH(\g)$-modules. Hence, $\tanH$ defines a functor from $\lfMod_\g^{\infty}$ to $\Mod_{\tanH(\g)}$, the category of modules over the Lie algebra $\tanH(\g)$.
\begin{Rem}
	If we regard $\g$ as an $\Linfty$ algebra, the category $\lfMod_\g^{\infty}$ indeed consists of locally finite  (or compactly supported) $\Linfty$ modules over $\g$.
	The local finiteness constraint is served for the convergence of the total   differential $d_{\tot}^M$ or of the morphism $f = \sum_{p\geqslant 0} f_p \colon M \to C(\g, N)$, when acting on any element in $M$.
	It is also closely related to the \textit{mildness condition}
	introduced by Buijs and Murillo in~\cite{BM}.
\end{Rem}	
Given any locally finite $\infty$-$\g$-module $(M, d_{\tot}^M)$, the pair $(C(\g,M), d_{\tot}^M)$ forms a  $C(\g)$-module. Meanwhile, any $\infty$-morphism $f \colon M \rightsquigarrow N$ of locally finite $\infty$-$\g$-modules gives rise to a morphism
\[
F = \id_{C(\g)} \otimes \sum_{k \geqslant 0} f_k \colon \big(C(\g,M), d_{\tot}^M\big) \to \big(C(\g, N), d_{\tot}^{N}\big)
\]
of  $C(\g)$-modules.
The two assignments define a functor
\begin{equation}\label{Eqt:SecondC}
\secondC(\g,-) \colon \quad \lfMod_\g^{\infty} \to \Mod_{C(\g)},
\end{equation}
which is called the Chevalley-Eilenberg functor of locally finite $\infty$-$\g$-modules.

\begin{prop}\label{prop: CE functor from infty modules}
	The  Chevalley-Eilenberg functor $\secondC(\g,-)$ of locally finite $\infty$-$\g$-modules in \eqref{Eqt:SecondC} is fully faithful. Moreover, it defines an equivalence of categories
	\[
	\secondC(\g,-) \colon \quad~\lfMod_\g^{\infty} \xrightarrow{\quad\simeq\quad} \fgsMod_{C(\g)},
	\]
	where $\fgsMod_{C(\g)} \subset \Mod_{C(\g)}$ is the subcategory consisting of finitely generated and semifree dg $C(\g)$-modules.
	Here,  a   $C(\g)$-module $\mathcal{M}$ is called \textit{semifree} if the underlying space $\mathcal{M}$ obtained by forgetting the differential is a free $C(\g)$-module.
\end{prop}
\begin{proof}
	It is clear that $\secondC(\g, -)$ is faithful. To see that it is also full, note that each morphism $F \in \Mod_{C(\g)}(\secondC(\g,M), \secondC(\g,{N}))$ can be decomposed by weight
	\[
	F = \sum_{k \geqslant 0} F_k \colon C^\bullet(\g,M) \to C^{\bullet+k}(\g,{N}).
	\]
	Since $F(m) = \sum_k F_k(m)$ is a finite sum for all $m \in M$, it thus defines an $\infty$-morphism $f$  from $M$ to ${N}$ with $F = \secondC(\g, f)$.
	
	Furthermore,  each object $\mathcal{M}$ in $\fgsMod_{C(\g)}$ has the form $(C(\g,M), d_{\tot}^M)$, where $C(\g,M)$ is the graded $C(\g)$-module freely generated by a finite dimensional graded vector space $M$, and $d_{\tot}^M$ is a differential on $C(\g,M)$ that has a decomposition by weight: $d_{\tot}^M = d_{\CE}+ \sum_{k \geqslant 0}d^M_k$, where $d^M_k \colon C^\bullet(\g,M) \to C^{\bullet+k}(\g,M)[1]$ is $C(\g)$-linear. These data give a {\LF} $\infty$-$\g$-module structure $d_{\tot}^M$ on $M$.
	Thus, $\secondC(\g,-) \colon \lfMod_\g^{\infty} \to \fgsMod_{C(\g)}$ is essentially surjective as well.
\end{proof}
Since an ordinary $\g$-module is also a {\LF} $\infty$-$\g$-module, thus, to each dg Lie algebra $\g$ there are associated three categories of modules
\[
\Module_\g \subset \mRepginfty \subset \lfMod_\g^\infty,
\]
where ${\mRepginfty}$ is the category comprised of $\g$-modules and  {$\infty$-morphisms} of $\g$-modules.
Applying the Chevalley-Eilenberg functor to these three categories, we obtain the following commutative diagram:
\[
\begin{tikzcd}
{\mRepginfty} \ar[r, hookrightarrow] \ar{dr}{{\firstC }(\g,-)}  &  \lfMod_\g^\infty \ar{d}[right]{\secondC(\g,-)}[left]{\simeq} \\
\Module_{\g}  \ar{u} \ar{r}{C(\g,-)} &    \fgsMod_{C(\g)} \ar[r, hookrightarrow] & \Mod_{C(\g)}.
\end{tikzcd}
\]
In the subsequent Section \ref{Section: homotopy theory}, we show that the category $\lfMod_\g^\infty$ carries almost a model category structure such that $\secondC(\g,-)$ preserves weak equivalences and fibrations  (whereas the category ${\mRepginfty}$ does not, see Remark \ref{Rem:pullbacksnotexist}).

\section{Locally finite infinity-modules as almost a model category}\label{Sec:Proof}
In~\cite{Vallette}, Vallette has shown that, for a dg operad $\mathcal{P}$, the category of  homotopy $\mathcal{P}$ algebras and their $\infty$-morphisms admits almost a model category structure (compared to a model category structure, only the axiom on limits and colimits is
not completely fulfilled).
Analogous to the situation of homotopy $\mathcal{P}$ algebras considered in~\cite{Vallette}, the existence of equalizers and coequalizers in category $\lfMod_\g^\infty$ is very subtle and requires further conditions. 
Therefore, the present category $\lfMod_\g^\infty$ does not admit a (closed) model category structure.
Following the same approach, we shall show that the category $\lfMod_\g^{\infty}$ of {\LF} $\infty$-$\g$-modules (introduced in Section \ref{Sec:LFinftygmodule}) also carries almost a model category structure (Theorem \ref{Thm: Repinftyg is fibrant}).
\subsection{The main theorem}
We first recall from \cites{Brown, Rogers, Vallette} some basic definitions in model category.
\begin{Def}
	Let $\mathfrak{C}$ be a category. A morphism $f$ is said to be a \textbf{retract} of $g$ if there exists a commutative diagram
	\[
	\begin{tikzcd}
	A \ar{d}{f} \ar{r} & C \ar{d}{g} \ar{r} & A \ar{d}{f} \\
	B \ar{r}            & D \ar{r}            & B
	\end{tikzcd}
	\]
	in $\mathfrak{C}$.
\end{Def}
\begin{Def}
	Let $i\colon A \rightarrow B$ and $p\colon X \rightarrow Y$ be morphisms in a category $\mathfrak{C}$. We say that $i$ has the \textbf{left lifting property} with respect to $p$  (or $p$ has the \textbf{right lifting property} with respect to $i$)  if every commutative square
	\[
	\begin{tikzcd}
	A \ar{d}{i} \ar{r}{f} & X \ar{d}{p} \\
	B \ar{r}{g}   & Y
	\end{tikzcd}
	\]
	can be completed to a commutative diagram
	\[
	\begin{tikzcd}
	A \ar{d}{i} \ar{r}{f} & X \ar{d}{p} \\
	B \ar{r}{g} \arrow[ur, dashed, "h"] & Y.
	\end{tikzcd}
	\]
	The morphism $h\colon B \rightarrow X$ is called a \textbf{lift}.
\end{Def}

\begin{Def}\label{Def: CFO}
	Let $\mathfrak{C}$ be a category with finite products, with terminal object $\ast \in \mathfrak{C}$, and equipped with two classes of morphisms which are called \textbf{weak equivalences}  and \textbf{fibrations}, respectively. A morphism is called a \textbf{trivial fibration} (or an acyclic fibration) if it is  both a weak equivalence and a fibration. We call $\mathfrak{C}$  a \textbf{category of fibrant objects} if the following axioms are satisfied:
	\begin{compactenum}
		\item Every isomorphism in $\mathfrak{C}$ is a trivial fibration.
		\item The class of weak equivalences satisfies the ``2 out of 3" criteria, i.e.,  if $f$ and $g$ are two composable morphisms in $\mathfrak{C}$, and any two of $f, g, g\circ f$ are weak equivalences, then so is the third one.
		\item The composition of   fibrations is again a fibration.
		\item All objects are fibrant, i.e., for any object $X \in \mathfrak{C}$, the unique morphism $X \to \ast$ is a fibration.
		\item The pullback of a  fibration exists, which is again a   fibration, i.e.,  given a diagram $Y \xrightarrow{g} Z  \xleftarrow{f} X$ in $\mathfrak{C}$ such that $f$ is a   fibration, then the pullback $X \times_Z Y$ exists, and the induced projection $X \times_Z Y \xrightarrow{p} Y$ is a   fibration. Moreover, if $f$ is a trivial fibration, then so is $p$.
		\item For any object $X \in \mathfrak{C}$, there is an associated  path object, i.e.,  an object $X^I\in \mathfrak{C}$ equipped with morphisms
		\[
		X \xrightarrow{s} X^I \xrightarrow{(\epsilon_0,\epsilon_1)} X \times X,
		\]
		such that $s$ is a weak equivalence, $(\epsilon_0, \epsilon_1)$ is a fibration, and their composition is the diagonal map $X\to X\times X$.
	\end{compactenum}
\end{Def}
\begin{Def}\label{Def: Almost model category}
	Let $\mathfrak{C}$ be a category of fibrant objects. Assume further that $\mathfrak{C}$ has an initial object $0$, and is equipped with another class of morphisms, called \textbf{cofibrations}. A morphism which is both a weak equivalence and a cofibration is called a \textbf{trivial cofibration} (or an acyclic cofibration). We call $\mathfrak{C}$ \textbf{almost a model category} if the following additional axioms are satisfied:
	\begin{compactenum}[$(a)$]
        \item This category $\mathfrak{C}$ admits finite products and coproducts\footnote{Note that only the existence of finite products is assumed in~\cite{Vallette}*{Theorem 4.2}. }.
		\item A retract of a weak equivalence (resp. fibration) is again a weak equivalence (resp. fibration).
		\item Any morphism $f$ in $\mathfrak{C}$ can be factored as $f = sj$, where $j$ is a cofibration and $s$ is a weak equivalence.
		\item Any morphism $f$ in $\mathfrak{C}$ can be factored as $f = pi$, where $p$ is a fibration, $i$ is a weak equivalence, and $i$ has the left lifting property with respect to all fibrations.
	\end{compactenum}
\end{Def}
We now focus on the category $\lfMod_\g^\infty$. Note that it is pointed, i.e., its final and initial objects are one and the same, namely the zero object in $\lfMod_\g^\infty$.
\begin{Def}\label{Def: Model structure for Repginfty}
	An  $\infty$-morphism  $(f\colon   M  \rightsquigarrow  {N})$  $  \in \lfMod_\g^{\infty}(M,N)$ is called
	\begin{itemize}
		\item [(N)] a \textbf{weak equivalence}, if it is a quasi-isomorphism of local finite   $\infty$-$\g$-modules, i.e., its leading component $f_0 \colon (M,d^M_0) \to (N,d^N_0)$ is a quasi-isomorphism of cochain complexes;
		\item [(F)] a \textbf{fibration}, if   $f_0  $ is an epimorphism, i.e., surjective;
		\item [(TF)] a \textbf{trivial fibration} if it is both a weak equivalence and a fibration;
		\item [(C)] a \textbf{cofibration}, if   $f_0$ is a monomorphism, i.e., injective;
		\item [(TC)]  a \textbf{trivial cofibration} if it is both a weak equivalence and a cofibration.
	\end{itemize}
\end{Def}

\begin{Thm}~\label{Thm: Repinftyg is fibrant}
The category $\lfMod_\g^\infty$ of \LF $\infty$-$\g$-modules admit finite products and finite coproducts.
\begin{compactenum}
		\item When equipped with the classes of weak equivalences and fibrations as defined above, the category $\lfMod_\g^{\infty}$ of {\LF} $\infty$-$\g$-modules  is a category of fibrant objects.
		\item Equipped with the class of cofibrations as defined above, $\lfMod_\g^{\infty}$ is almost a model category.
		\item Every object $M$ in $\lfMod_\g^{\infty}$ is fibrant and cofibrant in the sense that the morphism $M \to 0$ is a fibration and the morphism $0 \to M$ is a cofibration.
	\end{compactenum}
\end{Thm}

The proof is deferred to the subsequent subsections.
As an immediate application, we can explicitly describe the homotopy category $\Ho(\lfMod_\g^\infty)$  which is obtained from $\lfMod_\g^\infty$ by inverting  weak equivalences (see~\cites{Hovey, Quillenbook}).
To this end, we construct a path object explicitly for each object $M \in \lfMod_\g^{\infty}$.
First of all,  recall the model of commutative dg algebra   for the interval $I=[0,1]$,   also known as the \textit{Sullivan algebra}  in rational homotopy theory.

\begin{Def}[\cite{FHT}]
	Let $\Lambda(t,dt) = \k[t] \oplus \k[t]dt$ be the  commutative graded algebra freely generated by the basis $\{t,dt\}$ with degrees $\abs{t} = 0$ and $\abs{dt} = 1$, respectively, and let   $d$ be the   differential on $\Lambda(t,dt)$ extended from $t \mapsto dt$. The resulting commutative dg algebra $(\Lambda(t,dt), d)$ coincides with the space of polynomial differential forms on the $1$-simplex $I$, and is called the \textbf{Sullivan algebra}.
	Define two augmentations
	\[
	\epsilon_0, \epsilon_1 \colon \Lambda(t,dt) \to \k,\quad \text{by}\quad \epsilon_0(f(t)+g(t)dt) = f(0), ~\epsilon_1(f(t)+g(t)dt) = f(1). 
	\]
\end{Def}
The underlying cochain complex of the Sullivan algebra deformation retracts to its subcomplex $(J := \k \oplus \k t \oplus \k dt, d)$, known as the Dupont's contraction~\cites{Dupont, Getzler, Vallette}.

Given any object $M=(M,d_{\tot}^M) \in \lfMod_\g^{\infty}$, we are able to obtain a new object
\[
J(M) = \left(J(M):=J\otimes M, d_{\tot}^{J(M)} = \sum_{k\geqslant 0} d_k^{J(M)}\right)
\]
in $\lfMod_\g^{\infty}$  as described below:
\begin{itemize}
	\item For each $n \in \Z$,   $J(M)^n$ is the direct sum $ M^n \oplus M^n t \oplus M^{n-1}dt$, whose elements will be written as $(m_0,m_1,u)$ or $m_0+t m_1 + (dt) u$ for   $m_0,m_1 \in M^n, u \in M^{n-1}$.
	\item The differential $d_0^{J(M)}$ on $J(M) = \oplus_n J(M)^n$ is defined by
	\[
	d^{J(M)}_0(m_0+t m_1+ (dt) u) := d_0^M(m_0)+t d_0^M(m_1)  + dt(m_1 - d_0^M(u)).
	\]
	\item The {\LF} $\infty$-module structure $d^{J(M)}_k $ ($ k \geqslant 1$) is defined by
	\begin{align*}
	d_k^{J(M)}(m_0 + tm_1 + (dt) u) &:= d_k^M(m_0) + td_k^M(m_1) + (dt)d_k^M(u),
	\end{align*}
	for all $m_0+ tm_1 + (dt)u \in J(M)$.
\end{itemize}

We can verify that $J(M)$ is indeed a path object for $M$ due to the following fact.
\begin{prop}\label{prop: path object for Repginfty}The diagonal morphism $\Delta \colon M \to M \oplus M$ can be factored into
	\[
	\Delta \colon M \xhookrightarrow{s} J(M) \xrightarrow{(\epsilon_0,\epsilon_1)} M \oplus M ,
	\]
	where
	\begin{compactenum}
		\item the injection $s \colon M \hookrightarrow J(M)$ defined by $s(m) = (v,0,0)$ is a weak equivalence in $\lfMod_\g^{\infty}$;
		\item the map $(\epsilon_0,\epsilon_1) \colon J(M) \to M \oplus M$ defined by $(\epsilon_0,\epsilon_1)(m_0, m_1,  u) = (m_0, m_0 + m_1)$ is a fibration.
	\end{compactenum}
\end{prop}

\begin{Def}\label{Def: homotopy in Rep}
	Two {$\infty$-morphisms} $f,g\colon  M\rightsquigarrow N $ are said to be \textbf{homotopic}, written as $f \simeq g$, if there exists an $\infty$-morphism $H \colon M \rightsquigarrow J(N)$, called a \textbf{right homotopy}, such that the following diagram in   $\lfMod_\g^{\infty}$ commutes:
	\[
	\begin{tikzcd}
	& M \ar[rightsquigarrow]{ld}[left]{\overset{\mathlarger{f}}{\overset{~ \overset{~ }{~ }}{~\overset{~ }{~ } }}} \ar[d, rightsquigarrow, "H"] \ar[dr, rightsquigarrow, "g"] & \\
	N   & J(N) \ar{l}{\epsilon_0} \ar{r}[swap]{\epsilon_1} & N.
	\end{tikzcd}
	\]
\end{Def}

The homotopic relation is an equivalence relation for fibrant and cofibrant objects (see~\cite{Quillenbook}). Since every object in $\lfMod_\g^\infty$ is fibrant and cofibrant according to Theorem~\ref{Thm: Repinftyg is fibrant}, we have an induced category $\lfMod_\g^\infty{/_\sim}$,  whose objects are the same as $\lfMod_\g^\infty$ and whose morphisms are homotopy classes of $\infty$-morphisms in $\lfMod_\g^\infty$.
Also by~\cite{Quillenbook}, we can characterize the homotopy category $\Ho(\lfMod_\g^\infty)$ by an equivalence
\[
\Ho(\lfMod_\g^\infty)~ \simeq ~\lfMod_\g^\infty{/_\sim}.
\]
Note that the category $\Mod_{C(\g)}$ of $C(\g)$-modules is a model category~\cite{DAGX}.
The Chevalley-Eilenberg functor $\secondC(\g,-)\colon \lfMod_\g^{\infty} \rightarrow \Mod_{C(\g)}$ indeed preserves weak equivalences and fibrations. Clearly, $\secondC(\g,J(N))$ is a path object for $\secondC(\g,N)$ in $\Mod_{C(\g)}$. Applying   $\secondC(\g,-)$ to the commutative diagram in Definition~\ref{Def: homotopy in Rep}, we obtain a corresponding commutative diagram in   $\Mod_{C(\g)}$, and hence $\secondC(\g,H)$ defines a homotopy from the   $C(\g)$-module morphism $\secondC(\g,f)$ to $\secondC(\g,g)$.  We summarize this feature into the following proposition.
\begin{prop}\label{prop: homotopy invariance of CE functor}
	The Chevalley-Eilenberg functor $\secondC(\g,-) \colon \lfMod_\g^{\infty} \rightarrow \Mod_{C(\g)}$ maps homotopic $\infty$-morphisms in $\lfMod_\g^\infty$ to homotopic morphisms in the category $\Mod_{C(\g)}$.
\end{prop}

\subsection{Pullbacks of fibrations}

Inspired by Vallette~\cite{Vallette} and Rogers~\cite{Rogers}, we first propose and prove the following lemma on strictification of fibrations and cofibrations.
\begin{lem}\label{Lem: Strictify fibration}
	Let $\phi \colon (M, d_{\tot}^M) \rightsquigarrow (N, d_{\tot}^N)$ be a fibration of {\LF} $\infty$-$\g$-modules. Then there exists another {\LF} $\infty$-$\g$-module $(M, \widetilde{d_\tot^M})$ and an $\infty$-\textbf{isomorphism} $\psi \colon (M, \widetilde{d_{\tot}^M}) \rightsquigarrow (M,{d_\tot^M})$ with $\psi_0=\id_M$  such that
	\[
	\phi \psi \colon (M, \widetilde{d_\tot^M}) \rightsquigarrow (N, {d_\tot^N})
	\]
	is a strict fibration, i.e.,  $\phi\psi = (\phi\psi)_0 = \phi_0$ is a fibration and $(\phi\psi)_k = 0$ for all $k \geqslant 1$.
\end{lem}

\begin{proof}
	Since $\phi$ is a fibration, i.e.,  $\phi_0 \colon (M,d^M_0) \to (N,d^N_0)$ is a fibration of cochain complexes, we could choose a right inverse $\phi^{-1}_0 \colon N \to M$ such that $\phi_0 \circ \phi_0^{-1} = {\id_{N}}$.
	
	We now construct recursively a sequence of $C(\g)$-linear maps $\Psi_k \colon C^{\bullet}(\g,M) \to C^{\bullet+k }(\g,M), k \geqslant 0$ with $\Psi_0 = \id \colon C(\g,M) \to C(\g,M)$. Assume that $\Psi_n$ has been defined for all $0 \leqslant n \leqslant k$. Then we proceed to define $\Psi_{k+1} \colon C(\g,M) \to C^{\bullet+k+1}(\g,M)$ by
	\[
	\Psi_{k+1} = - \sum_{\substack{0 \leqslant p \leqslant k, 1\leqslant q \leqslant k+1 \\ p+q = k+1}} \Phi_0^{-1}\Phi_q\Psi_p,
	\]
	where $\Phi_q = \id_{C(\g)}\otimes \phi_q \colon C(\g,M) \to C^{\bullet+q}(\g,N)$
	and $\Phi_0^{-1} \colon C(\g,N) \to C(\g,M)$ is the $C(\g)$-linear extension of the chosen right inverse $\phi_0^{-1}$.
	This construction indeed defines an isomorphism $\Psi = \sum_{k \geqslant 0} \Psi_k$ of the $C(\g)$-module $\secondC(\g,M)$.
	Let $\widetilde{d^M_{\tot}} := \Psi^{-1} \circ d_M^{\tot} \circ \Psi$ be the total differential induced from this isomorphism $\Psi$.
	Then $\psi = \Psi\mid_{M} \colon (M, \widetilde{d_{\tot}^M}) \rightsquigarrow (M,d_{\tot}^M)$ is an isomorphism of locally finite $\infty$-$\g$-modules.
	
	Finally, note that $(\Phi\Psi)_0 = \Phi_0 \Psi_0 = \Phi_0$ and
	\begin{eqnarray*}
		(\Phi\Psi)_{k+1} &=& \Phi_0\Psi_{k+1} + \sum_{\substack{0 \leqslant p \leqslant k, 1\leqslant q \leqslant k+1 \\ p+q = k+1}}\Phi_q\Psi_p \\
		&=& -\Phi_0 \Bigl(\sum_{\substack{0 \leqslant p \leqslant k, 1\leqslant q \leqslant k+1 \\ p+q = k+1}} \Phi_0^{-1}\Phi_q\Psi_p\Bigr) + \sum_{\substack{0 \leqslant p \leqslant k, 1\leqslant q \leqslant k+1 \\ p+q = k+1}}\Phi_q\Psi_p = 0,
	\end{eqnarray*}
	for all $k \geqslant 0$. Thus $\phi\psi$ is a strict fibration, as desired.
\end{proof}

In a similar manner, one can prove the following lemma about strictification on cofibrations.
\begin{lem}\label{lem: strictify cofibration}
	Let $\phi \colon (M,d_\tot^M) \rightsquigarrow (N, {d_\tot^N })$ be a cofibration of {\LF} $\infty$-$\g$-modules. Then there exists another {\LF} $\infty$-module $(N,\widetilde{d^N_\tot})$ and an $\infty$-\textbf{isomorphism} $\psi \colon (N,{d_\tot^N }) \rightsquigarrow (N,\widetilde{d_{\tot}^N})$  with $\psi_0=\id_N$  such that
	\[
	\psi\phi \colon (M, {d_\tot^M}) \rightsquigarrow (N,\widetilde{d_{\tot}^N})
	\]
	is a strict cofibration, i.e., $\psi\phi = (\psi\phi)_0 = \phi_0$ is a cofibration and $(\psi\phi)_k = 0$ for all $k \geqslant 1$.
\end{lem}

\begin{prop}\label{prop: pullback of strict fibrations}
	Suppose that $\phi \colon M \rightsquigarrow M^{\prime\prime}$ is a strict fibration in $\lfMod_\g^{\infty}$ and that $\psi \colon M^\prime \rightsquigarrow M^{\prime\prime}$ is an $\infty$-morphism in $\lfMod_\g^{\infty}$.
	Let $(\widetilde{M}, d^{\widetilde{M}}_0)$ denote the pullback of the diagram of cochain complexes
	\begin{equation}\label{Eq: diagram in Chk}
	\begin{tikzcd}
	(\widetilde{M}, d^{\widetilde{M}}_0) \arrow[r] \arrow[d]
	&   (M, d^M_0) \arrow[d, "\phi_0"]  \\
	(M^\prime, d^{M^\prime}_0 ) \arrow[r,  "\psi_0"]  &  (M^{\prime\prime }, d^{M^{\prime\prime}}_0).
	\end{tikzcd}
	\end{equation}
	\begin{compactenum}
		\item The pullback square \eqref{Eq: diagram in Chk} in the category ${\mathbf{Ch}}(\k)$ of cochain complexes of vector spaces can be lifted to a commutative square
		\begin{equation}\label{Eq: CD in pullback}
		\begin{tikzcd}
		(\widetilde{M},  {d^{\widetilde{M}}_{\tot}}) \arrow[r, rightsquigarrow] \arrow[d]
		&   (M, d^M_{\tot}) \arrow[d, "\phi"]  \\
		(M^\prime, d_{\tot}^{M^\prime}) \arrow[r, rightsquigarrow, "\psi"]  &  (M^{\prime\prime }, d_{\tot}^{M^{\prime\prime}})
		\end{tikzcd}
		\end{equation}
		in the category $\lfMod_\g^{\infty}$. Note that the two horizontal arrows are $\infty$-morphisms, whereas the two vertical arrows are strict morphisms.
		\item The pullback operation is compatible with the Chevalley-Eilenberg functor $\secondC(\g,-)$ in the following sense:
		When   $\secondC(\g,-)$ is applied to Diagram \eqref{Eq: CD in pullback}, the resulting diagram
		\begin{equation*}
		\begin{tikzcd}
		\secondC(\g,\widetilde{M}) \arrow[rr] \arrow[d]
		& & C(\g,M) \arrow[d, "\id_{C(\g)}\otimes\phi"]  \\
		C(\g, M^\prime) \arrow[rr,  "  \id_{C(\g)}\otimes\psi"]  & & C(\g,M^{\prime\prime})
		\end{tikzcd}
		\end{equation*}
		is a pullback in the category of   $C(\g)$-modules.
	\end{compactenum}
\end{prop}
\begin{proof}
	To prove Statement (1), we first find  $\widetilde{M}$ as a graded vector space.
	Since $\phi = \phi_0 \colon M \to M^{\prime\prime}$ is a fibration, we may choose a right inverse of the linear map $\phi$, i.e., a linear map $\phi^{-1} \colon M^{\prime\prime } \to M$ such that $\phi \phi^{-1} = \id_{ M^{\prime\prime }}$.
	Then we define $\widetilde{M} := M^{\prime}\oplus \ker \phi \subset M^\prime \oplus M$, and form a pullback diagram of graded vector spaces
	\[
	\begin{tikzcd}
	\widetilde{M} \arrow[rr, "\pr + \phi^{-1}\psi_0\pr^\prime"] \arrow[d, "\pr^\prime"]
	  & & M \ar[d, "\phi"] \\
	M^\prime \arrow[rr, "\psi_0"] & & M^{\prime\prime},
	\end{tikzcd}
	\]
	where $\pr \colon M \oplus M^\prime \to M$ and $\pr^\prime \colon M \oplus M^\prime \to M^\prime$ are    projections.
	
	Consider the $C(\g)$-linear map $H \colon \secondC(\g, M^\prime \oplus M) \to \secondC(\g, M^\prime \oplus M)$ defined by
	\[
	H(m^\prime, m) = (m^\prime, m + \Phi^{-1}\Psi(m^\prime)),
	\]
	for all homogeneous $m^\prime \in M^\prime$ and $m \in M$, where $\Phi^{-1}$ is the $C(\g)$-linear extension of $\phi^{-1}$, and $\Psi = \secondC(\g,\psi) \colon \secondC(\g, M^\prime) \to \secondC(\g, M^{\prime\prime})$ is a morphism of $C(\g)$-modules.
	
	We now define a {\LF} $\infty$-$\g$-module structure on $\widetilde{M}$ via specifying a differential $d^{\widetilde{M}}_{\tot}$ on $C(\g,\widetilde{M})$.
	To do so, we first define a differential $\widetilde{d_{\tot}} \colon C(\g, M^\prime \oplus M) \to C(\g, M^\prime \oplus M)[1]$ by
	\[
	\widetilde{d_{\tot}}(m^\prime, m) = (d_{\tot}^\prime m^\prime, d_{\tot}m + d_{\tot}\Phi^{-1}\Psi(m^\prime) - \Phi^{-1}\Psi(d_{\tot}^\prime m^\prime)),
	\]
	where, to simplify notations, we use $d_{\tot}=d^M_{\tot} \colon C(\g,M) \to C(\g,M)[1]$ and $d_{\tot}^\prime=d_{\tot}^{M^\prime} \colon C(\g, M^\prime) \to C(\g, M^\prime)[1]$ to denote the differentials of the {\LF} $\infty$-modules $M$ and $M^\prime$, respectively.
	Since $m \in \ker(\phi)$, we have
	\[
	\Phi(d_{\tot}m + d_{\tot}\Phi^{-1}\Psi(m^\prime) - \Phi^{-1}\Psi(d_{\tot}^\prime m^\prime)) = d_{\tot}^{\prime\prime}\phi(m) + d_{\tot}^{\prime\prime}\Psi(m^\prime) - d_{\tot}^{\prime\prime}\Psi(m^\prime) = 0.
	\]
	Thus, $\widetilde{d_{\tot}}$ restricts onto a differential  $d^{\widetilde{M}}_{\tot}$ on $\secondC(\g,\widetilde{M})$.
	Furthermore, both the two maps
	\begin{align*}
	\pr \circ  H &\colon C(\g,\widetilde{M}) \to C(\g,M), & \pr^\prime \circ H &\colon  C(\g,\widetilde{M}) \to C(\g, M^\prime)
	\end{align*}
	are cochain maps, i.e.,
	\begin{align*}
	\pr \circ H \circ \widetilde{d_{\tot}} &= d_{\tot} \circ \pr \circ H \colon C(\g,\widetilde{M}) \to C(\g,M), \\
	\pr^\prime \circ H \circ \widetilde{d_{\tot}} &=   d_{\tot}^\prime \circ \pr^\prime \circ H \colon  C(\g,\widetilde{M}) \to C(\g, M^\prime).
	\end{align*}
	It suffices to verify the first identity. Indeed, for any $(m^\prime, m) \in \widetilde{M}$,  we have
	\begin{align*}
	(\pr \circ H \circ \widetilde{d_{\tot}})(m^\prime, m) &= d_{\tot}m + d_{\tot}\Phi^{-1}\Psi(m^\prime) - \Phi^{-1}\Psi(d_{\tot}^\prime m^\prime) + \Phi^{-1}\Psi(d_{\tot}^\prime m^\prime) \\
	&= d_{\tot}m + d_{\tot}\Phi^{-1}\Psi(m^\prime) = (d_{\tot} \circ \pr \circ H)(m^\prime, m).
	\end{align*}
	Finally, we verify that the following diagram in the category $\Mod_{C(\g)}$ of  $C(\g)$-modules
	\begin{equation}\label{Eq: CD in ModCg}
	\begin{tikzcd}
	\secondC(\g,\widetilde{M}) \ar[r, "\pr \circ H"] \ar{d}[left]{\pr^\prime \circ H}   & \secondC(\g,M) \ar[d, "\Phi"] \\
	\secondC(\g,{N}) \ar[r, "\Psi"] & \secondC(\g,M^{\prime\prime})
	\end{tikzcd}
	\end{equation}
	commutes, where $\Phi = \secondC(\g,\phi) \colon \secondC(\g,M) \to \secondC(\g, M^{\prime\prime})$ is the $C(\g)$-linear extension of $\phi \colon M \to M^{\prime\prime}$ since $\phi = \phi_1$ is strict.
	In fact,  for any $(m^\prime, m) \in \widetilde{M} = M^\prime \oplus \ker(\phi)$, we have
	\[
	(\Phi\circ \pr \circ H)(m^\prime, m) = (\Phi \circ \pr)(m^\prime, m + \Phi^{-1}\Psi(m^\prime)) = \Phi(m) + \Psi(m^\prime) = \Psi(m^\prime),
	\]
	and on the other hand,
	\[
	(\Psi \circ \pr^\prime \circ H)(m^\prime, m) = (\Psi \circ \pr^\prime)(m^\prime, m + \Phi^{-1}\Psi(m^\prime)) = \Psi(m^\prime).
	\]
	This proves that Diagram~\eqref{Eq: CD in ModCg} commutes.
	By Proposition~\ref{prop: CE functor from infty modules}, we obtain a commutative Diagram~\eqref{Eq: CD in pullback} that lifts the pullback square  ~\eqref{Eq: diagram in Chk}  in the category ${\mathbf{Ch}}(\k)$, as desired.
	
	Now we prove the second statement. Let $\widetilde{\mathcal{M}}$ be the pullback of the following diagram in the category $\Mod_{C(\g)}$:
	\[
	\begin{tikzcd}
	\widetilde{\mathcal{M}} \ar{r}{\Pr} \ar{d}[left]{\Pr^\prime}  & C^{\infty}(\g, M) \ar{d}{\Phi} \\
	\secondC(\g, M^\prime) \ar{r}{\Psi}        &     \secondC(\g,M^{\prime\prime}),
	\end{tikzcd}
	\]
	where $\Phi = \secondC(\g,\phi) \colon \secondC(\g,M) \to \secondC(\g,M^{\prime\prime})$ is a fibration in $\Mod_{C(\g)}$.
	Note that the  $C(\g)$-module $\widetilde{\mathcal{M}}$, by definition, is naturally identified with the submodule $\ker(\Phi\Pr - \Psi\Pr^\prime)$ of $\secondC(\g, M^\prime \oplus M)$.
	
	Consider the $C(\g)$-linear map $J \colon C(\g, M^\prime \oplus M) \to C(\g, M^\prime \oplus M)$ defined by
	\[
	J(m^\prime, m)  = (m^\prime, m - \Phi^{-1}\Psi(m^\prime)),
	\]
	where $\Phi^{-1}$ is the chosen right inverse of the fibration $\Phi$.
	Meanwhile, since
	\begin{align*}
	\widetilde{d_{\tot}}J(m^\prime, m) &= \widetilde{d_{\tot}}(m^\prime, m - \Phi^{-1}\Psi(m^\prime)) \\
	&= (d_{\tot}^\prime m^\prime, d_{\tot}(v - \Phi^{-1}\Psi(m^\prime)) + d_{\tot}\Phi^{-1}\Psi(m^\prime) - \Phi^{-1}\Psi(d_{\tot}^\prime m^\prime)) \\
	&= (d_{\tot}^\prime m^\prime, d_{\tot}m - \Phi^{-1}\Psi(d_{\tot}^\prime m^\prime)) \\
	&= J(d_{\tot}^\prime m^\prime, d_{\tot}m) = J d_{\tot}^{M^\prime \oplus M }(m^\prime, m),
	\end{align*}
	it follows that $J$ is a morphism of  $C(\g)$-modules from $(C(\g, M^\prime \oplus M), d_{\tot}^{M^\prime \oplus M})$ to $(C(\g, M^\prime \oplus M), \widetilde{d_{\tot}})$.
	Since for any $(m^\prime, m) \in \widetilde{\mathcal{M}} = \ker(\Phi\Pr - \Psi\Pr^\prime)$, we have $\Psi(m^\prime) = \Phi(m)$, and
	\[
	\Phi \Pr J(m^\prime, m) = \Phi(m - \Phi^{-1}\Psi(m^\prime)) = \Phi(m) - \Psi(m^\prime) = 0,
	\]
	we have
	\[
	\Im J\mid_{\widetilde{\mathcal{M}}} \subset \secondC(\g, M^\prime) \oplus \ker(\Phi) = \secondC(\g, \widetilde{M}).
	\]
	Note further that $HJ = JH =\id_{C(\g, M^\prime \oplus M)}$. Hence, we have an isomorphism of   $C(\g)$-modules
	\[
	J \colon \widetilde{\mathcal{M}} \xrightarrow{\cong} \secondC(\g, \widetilde{M}).
	\]
\end{proof}

The following corollary is a key fact.
\begin{Cor}\label{Cor: pullback of fibration in Repginfty}
	Suppose that $\phi \colon (M, d^M_{\tot}) \rightsquigarrow ( M^{\prime\prime }, d_{\tot}^{M^{\prime\prime}})$ is a  fibration in $\lfMod_\g^{\infty}$, and that\\ $\psi \colon
	(M^\prime, d_{\tot}^{M^ \prime} ) \rightsquigarrow (M^{\prime\prime }, d_{\tot}^{M^{\prime\prime}})$ is an arbitrary morphism between {\LF} $\infty$-$\g$-modules. Then the pullback of the diagram
	\[
	\begin{tikzcd}
	(\widetilde{M},  {d^{\widetilde{M}}_{\tot}}) \arrow[r, rightsquigarrow] \arrow[d,rightsquigarrow,"P"]
	&   (M, d^M_{\tot}) \arrow[d,rightsquigarrow, "\phi"]  \\
	(M^\prime, d_{\tot}^{M^ \prime}) \arrow[r, rightsquigarrow, "\psi"]  &  (M^{\prime\prime }, d_{\tot}^{M^{\prime\prime}})
	\end{tikzcd}
	\]
	exists in the category $\lfMod_\g^{\infty}$, and the morphism $P$ induced by the pullback of $\phi$ along $\psi$ is a fibration.
	If $\phi$ is a trivial fibration, then so is $P$.
\end{Cor}
\begin{proof}
	It suffices to prove the statement under the assumption that $\phi$ is a fibration. If $\phi$ is a trivial fibration,   all the following proof should be adapted.
	
	By Lemma~\ref{Lem: Strictify fibration}, there is a {\LF} $\infty$-module $\widehat{M} = (M,  {d^{\widehat{M}}_{\tot}})$ and an isomorphism $i \colon (M, {d^{\widehat{M}}_{\tot}}) \rightsquigarrow (M, d^M_\tot)$ such that $\phi i \colon (M, {d^{\widehat{M}}_{\tot}} ) \to (M^{\prime\prime}, d_{\tot}^{M^{\prime\prime}})$ is a strict fibration with $\phi i = (\phi i)_0 = \phi_0$.
	It follows from Proposition~\ref{prop: pullback of strict fibrations} that there exists a pullback diagram in the category $\lfMod_\g^{\infty}$
	\[
	\begin{tikzcd}
	(\widetilde{M}, {d^{\widetilde{M}}_{\tot}})  \ar[r, rightsquigarrow] \ar{d}{\widetilde{\phi i}}    & (M, {d^{\widehat{M}}_{\tot}} ) \ar[d, two heads, "\phi i"]  \\
	(M^\prime, d_{\tot}^{M^ \prime} ) \ar[r, rightsquigarrow, "\psi"] & ( M^{\prime\prime }, d_{\tot}^{M^{\prime\prime}}),
	\end{tikzcd}
	\]
	such that $\widetilde{\phi i}$ is a fibration as well. In what follows, we simplify notations by $\widetilde{d_{\tot}}:=d_{\tot}^{\widetilde{M}}$, $\widehat{d_{\tot}}:=d_{\tot}^{\widehat{M}}$,  $ {d_{\tot}}:=d_{\tot}^{ {M}}$,  $ {d^\prime_{\tot}}:=d_{\tot}^{ {M^\prime}}$, and  $ {d^{\prime\prime}_{\tot}}:=d_{\tot}^{ {M^{\prime\prime}}}$.
	
	Let $\Phi, \Psi, I$ be the $C(\g)$-linear extensions of $\phi, \psi, i$, respectively. Denote by $\widetilde{\mathcal{M}}$ the pullback dg $C(\g)$-module of $\Phi$ along $\Psi$.
	According to the second statement of Proposition~\ref{prop: pullback of strict fibrations} and the pasting lemma for pullbacks, we have the following commutative diagram in the category $\Mod_{C(\g)}$:
	\[
	\begin{tikzcd}
	(\secondC(\g, \widetilde{M}), \widetilde{d_{\tot}}) \ar{r} \ar{d}[right]{\cong}[left]{\widetilde{I}} & (\secondC(\g,M), \widehat{d_{\tot}}) \ar{d}[right]{I}[left]{\cong}  \\
	\widetilde{\mathcal{M}} \ar{r} \ar{d}[left]{\widetilde{\Phi}} & (\secondC(\g, M),d_{\tot}) \ar{d}{\Phi} \\
	(\secondC(\g, M^\prime),d_{\tot}^\prime) \ar{r}{\Psi}   &  (\secondC(\g, M^{\prime\prime}), d_{\tot}^{\prime\prime}).
	\end{tikzcd}
	\]
	Note that $\widetilde{I}$ is an isomorphism of  $C(\g)$-modules. It follows that $\widetilde{\mathcal{M}}$ is freely generated by $\widetilde{I}(\widetilde{M})$, thus lies in the image of the Chevalley-Eilenberg functor $\secondC(\g,-)$.
	Meanwhile, since $\widetilde{\Phi}\widetilde{I}$ is a fibration and $\widetilde{I}$ is an isomorphism, it follows that $\widetilde{\Phi}$ is a fibration.
	Finally, using Proposition~\ref{prop: CE functor from infty modules}, we see that the pullback $P$ of $\phi$ along $\psi$ in the category $\lfMod_\g^{\infty}$ exists, and is a fibration as well.
\end{proof}

\begin{Rem}\label{Rem:pullbacksnotexist}
	Note that if we only consider the full subcategory ${\mRepginfty}$ of  $\lfMod_\g^\infty$, the strictification on fibrations (resp. cofibrations) in Lemma~\ref{Lem: Strictify fibration} (resp. Lemma~\ref{lem: strictify cofibration}) does not hold any more. Meanwhile, Proposition~\ref{prop: pullback of strict fibrations} and Corollary~\ref{Cor: pullback of fibration in Repginfty} do not hold for ${\mRepginfty}$. Thus,  in general, pullbacks of (trivial) fibrations do not exist in the category ${\mRepginfty}$.
\end{Rem}

\subsection{The lifting property}
Before proving that the category $\lfMod_\g^{\infty}$ is almost a model category, we establish the following left lifting property.
\begin{lem}\label{lem: LLP}
	Suppose that we have a commutative diagram in the category $\lfMod_\g^{\infty}$
	\[
	\begin{tikzcd}
	A \ar[r, rightsquigarrow, "f"] \ar[d, rightsquigarrow, tail, "j"] & B \ar[d, rightsquigarrow, two heads, "p"] \\
	C \ar[r, rightsquigarrow, "g"] & D,
	\end{tikzcd}
	\]
	where $j$ is a cofibration, and $p$   a fibration.
	If either $j$ or $p$ is also a weak equivalence, then $j$ has the left lifting property with respect to $p$, i.e.,  there exists an    $\infty$-morphism $l \colon C \to B$ fitting into the following commutative diagram:
	\begin{equation}\label{Eq: diagram in LLP}
	\begin{tikzcd}
	A \ar[r, rightsquigarrow, "f"] \ar[d, rightsquigarrow, tail, "j"] & B \ar[d, rightsquigarrow, two heads, "p"] \\
	C \ar[r, rightsquigarrow, "g"]  \ar[ur, rightsquigarrow, "l"] & D.
	\end{tikzcd}
	\end{equation}
\end{lem}

\begin{proof}
	By Lemmas~\ref{Lem: Strictify fibration} and  ~\ref{lem: strictify cofibration}, we may assume that both $p$ and $j$ are strict.  We proceed by induction on the weight $n$ of the lifting $l_n \colon C \to C^n(\g, B), n \geqslant 0$.
	
	The left lifting property of the projective model structure on the category ${\mathbf{Ch}}(\k)$ of cochain complexes of $k$-vector spaces~\cite{Hovey} implies that there exists a cochain map $l_0 \colon C \to B$ fitting into the following diagram in the category ${\mathbf{Ch}}(\k)$:
	\[
	\begin{tikzcd}
	(A,d_0^A) \ar{r}{f_0} \ar[d, tail, "j"] & (B,d_0^B) \ar[d, two heads, "p"] \\
	(C,d_0^C) \ar{r}{g_0} \ar{ur}{l_0} & (D,d_0^D).
	\end{tikzcd}
	\]
	Now assume that for some integer $k \geqslant 1$, we have constructed
	multilinear maps $l_i \colon C \to C^i(\g,B)$ for all $1 \leqslant i \leqslant k-1$, such that Diagram~\eqref{Eq: diagram in LLP} commutes up to weight $k-1$. We are going to construct a $C(\g)$-linear map
	\[
	l_{k} \colon C^\bullet(\g,C) \to C^{\bullet+k}(\g,B),
	\]
	such that Diagram~\eqref{Eq: diagram in LLP} commutes up to weight $k$, i.e.,
	\begin{align}\label{Eq: LLP1}
	l_{k}(j(a)) &= f_{k}(a),  \\
	p(l_{k}(c)) &= g_{k}(c),  \label{Eq: LLP2} \\
	\sum_{i+j=k}l_{j}(d^C_{i}(c)) &=  \sum_{r+s = k} d^B_{r}(l_{s}(c)), \label{Eq: LLP3}
	\end{align}
	for all homogeneous $a \in A$, and $c \in C$.
	Since $j$ is a trivial cofibration and $p$ is a fibration, we may choose a retraction $j^{-1} \colon C \to A$ such that $j^{-1}j = \id_A$, and a section $p^{-1} \colon D \to B$ such that $pp^{-1} = \id_D$ in the category of graded vector spaces.
	Consider a $C(\g)$-linear map $\tilde{l}_{k} \colon C^\bullet(\g,C) \to C^{\bullet+k}(\g,B)$ defined by
	\begin{align*}
	\tilde{l}_{k}(c) &= f_{k}(j^{-1}(c)) - (p^{-1}p f_{k})(j^{-1}(c)) + p^{-1}(g_{k}(c)).
	\end{align*}
	It is clear that $\tilde{l}_k$ satisfies Equations~\eqref{Eq: LLP1} and~\eqref{Eq: LLP2}. However, this map $\tilde{l}_{k}$, together with $l_0,\cdots,l_{k-1}$, does not solve Equation~\eqref{Eq: LLP3}. To solve this equation, we rewrite it as follows:
	\begin{align*}
	d_0^{BC}(l_k)(c) = d_{\geqslant 1}^{BC}(l_{k-1})(c),
	\end{align*}
	where
	\begin{align*}
	d_0^{BC}(l_k)(c) &:= (d_0 \otimes \id_B)(l_k(c)) + (\id_{C^k(\g)} \otimes d_0^B)(l_{k}(c)) - l_{k}(d_0^C(c)), \\
	d_{\geqslant 1}^{BC}(l_{k-1})(c) &:= \sum_{i=1}^{k-1}l_{k-i}(d_i^C(c)) - \sum_{j=1}^{k-1}(\id_{C^{k-j}(g)} \otimes d^B_j)(l_{k-j}(c)) - (d_1 \otimes \id_B)(l_{k-j}(c)).
	\end{align*}
	By our induction assumption, the multilinear map $d_0^{BC}(\tilde{l}_{k}) - d_{\geqslant 1}^{BC}(l_{k-1})$ is indeed $d_0^{BC}$-closed, and satisfies the following conditions:
	\begin{align*}
	(d_0^{BC}(\tilde{l}_{k}) - d_{\geqslant 1}^{BC}(l_{k-1}))(j(a)) &= (d_0^{BA}(f_{k}) -\bar{f}_{k})(a) = 0, \\
	p((d_0^{BC}(\tilde{l}_{k}) - d_{\geqslant 1}^{BC}(l_{k-1}))(c)) &= (d_0^{DC}(g_{k}) - \bar{g}_{k})(c) = 0,
	\end{align*}
	since both $f$ and $g$ are {$\infty$-morphisms}.
	Hence, the linear map $d_0^{BC}(\tilde{l}_{k}) - d_{\geqslant 1}^{BC}(l_{k-1})$ factors through a map
	\[
	\lambda_{k} \colon \coker(j) \rightarrow C^k(\g,\ker(p)),
	\]
	which is also $d_0^{BC}$-closed. Now if either $j$ or $p$ is a weak equivalence, then either $\coker(j)$ or $\ker(p)$ is acyclic. It follows that $\lambda_{k}$ is $d_0^{BC}$-exact. Thus, there exists a multilinear map $\theta_{k} \colon \coker(j) \to C^{k}(\g,\ker(p))$ such that $\lambda_{k} = d_0^{BC}(\theta_{k})$.
	Now define the $C(\g)$-linear map $l_{k} \colon C^\bullet(\g, C) \to C^{\bullet+k}(\g, B)$ by
	\[
	l_{k}(c) = \tilde{l}_{k}(c) - i(\theta_{k}(q(c))),
	\]
	where $q \colon C \twoheadrightarrow \coker(j)$ is the standard projection and $i \colon \ker(p) \hookrightarrow B$ is the standard inclusion. It is clear that this map $l_{k}$ satisfies Equations~\eqref{Eq: LLP1} $\sim$ \eqref{Eq: LLP3}, which concludes the proof of existence of a morphism $l\colon C \rightsquigarrow B$ of {\LF} $\infty$-modules.
\end{proof}

\subsection{Proof of Theorem~\ref{Thm: Repinftyg is fibrant}}\label{Sec:ProofofModelCat}

Note that the category $\lfMod_\g^{\infty}$ admits finite products and coproducts that are given by the obvious direct sum operation.
Therefore, to prove that $\lfMod_\g^{\infty}$ is a category of fibrant objects, it suffices to check that the category $\lfMod_\g^{\infty}$, together with the classes of weak equivalence and fibrations defined in Definition~\ref{Def: Model structure for Repginfty}, satisfies the axioms in Definition~\ref{Def: CFO}.
In fact, Axioms $(1)$ $\sim$ $(4)$ are obvious by definition of fibrations and weak equivalences. Axiom $(5)$ and Axiom $(6)$ follow from Corollary~\ref{Cor: pullback of fibration in Repginfty} and Proposition~\ref{prop: path object for Repginfty}, respectively.

To prove that $\lfMod_\g^{\infty}$ is also  almost a model category, we need to verify Axioms $(a)$ $\sim$ $(c)$ in Definition~\ref{Def: Almost model category}. Axiom $(a)$ holds obviously.
So we only handle Axioms $(b)$ and $(c)$.

Assume that $\phi \colon (M, d_{\tot}^M) \rightsquigarrow (N, d_{\tot}^N)$ is a morphism in the category $\lfMod_\g^{\infty}$.
Consider the mapping cone
\[
\Cone(M) = (M \oplus M [1], d_{\Cone(M)})
\]
of the identity cochain map ${\id_{M}} \colon (M, d_0^M) \to (M, d_0^M)$.
Equip $\Cone(M)$ with the trivial $\infty$-$\g$-module structure, i.e., $d_0^{\Cone(M)} = d^{\Cone(M)}$ and $d_k^{\Cone(M)} = 0, k \geqslant 1$.
It follows from a straightforward computation that the canonical inclusion $i_0 \colon (M, d_0^M) \hookrightarrow \Cone(M)$ extends to the $\infty$-morphism $i \colon M \rightsquigarrow \Cone(M)$ in the category $\lfMod_\g^{\infty}$ defined by
\[
i_k \colon M \to C^k(\g,\Cone(M)), \quad i_k(m) = (0, d^M_k(m)),
\]
for all $k\geqslant 1$ and $m \in M$.
Moreover, $i \colon M \to \Cone(M)$ is a cofibration. Consider the product $\infty$-morphism $j$ of the two {$\infty$-morphisms} $\Phi \colon M \rightsquigarrow N$ and $i \colon M \rightsquigarrow \Cone(M)$:
\[
\begin{tikzcd}
N \prod \Cone(M) \ar[r, rightsquigarrow, two heads, "s"] \ar[d, rightsquigarrow] & N \\
\Cone(M)  & M. \ar[l, rightsquigarrow, tail, "i"] \ar[u, rightsquigarrow, "\Phi"] \ar[lu, rightsquigarrow, tail, "j"]
\end{tikzcd}
\]
Since the underlying cochain complex in $N \prod \Cone(M)$ is the direct sum $N \oplus \Cone(M)$, it follows that $j$ is a cofibration and $s$ is a trivial fibration. So the factorization $\Phi = sj$ satisfies Axiom $(b)$.

Finally, we check Axiom $(c)$. Let $D := \coker(\phi_0)$ be the cokernel of the cochain map $\phi_0 \colon (M, d_0^M) \to (N,d^N_0)$. Consider the mapping cone $C := D \oplus D[1]$ of the identity $\id_D$ as a trivial {\LF} $\infty$-$\g$-module. The product of $(M,{d_0^M})$ and the trivial module $C$ has the form
\[
M \prod C = (M \oplus C, d_0^{M \prod C}).
\]
Let $i \colon M \rightarrowtail M \prod C$ be the natural strict monomorphism of {\LF} $\infty$-modules defined by ${\id_{M}} \oplus 0$, and let
\[
p_0 \colon (M \oplus C, d_0^{M \oplus C}) \twoheadrightarrow (W \cong \Im(\phi_0) \oplus D, d^N_0)
\]
be the cochain map defined by the sum $\phi_0 \oplus \pr_D$ of the cochain map $\phi_0$ and standard projection $\pr_D \colon C \twoheadrightarrow D$. It is clear that $p_0$ provides us a lifting of the following diagram in the category ${\mathbf{Ch}}(\k)$ of cochain complexes of vector spaces:
\[
\begin{tikzcd}
(M, d_0^M) \ar{r}{\phi_0} \ar[tail]{d}[right]{i}[left]{\simeq} & (N,d^N_0)  \ar[d, two heads] \\
(M \oplus C, d_0^{M \oplus C}) \ar{r} \ar[ur, two heads, dashed, "p_0"] & 0.
\end{tikzcd}
\]
By a similar argument in the proof of Lemma~\ref{lem: LLP}, we see that there exists {an $\infty$-morphism} $p \in \lfMod_\g^{\infty}(M \prod C, N)$ extending the cochain map $p_0$, such that the following diagram in the category $\lfMod_\g^{\infty}$ commutes:
\[
\begin{tikzcd}
M \ar[r, rightsquigarrow, "\phi"] \ar[tail]{d}[right]{i}[left]{\simeq} & N \ar[d, two heads] \\
M \prod C \ar[r] \ar[ur, rightsquigarrow, two heads, "p"] & 0.
\end{tikzcd}
\]
Hence, $\phi$ can be factorized as $pi$, where $i$ is a trivial cofibration, and thus has the left lifting property with respect to all fibrations by Lemma~\ref{lem: LLP}.

Finally, the assertion that every object in $\lfMod_\g^\infty$ is both fibrant and cofibrant is obvious. The proof is completed. \qquad \qquad    \qquad \qquad\qquad \qquad \qquad\qquad \qquad \qquad\qquad \qquad \qquad \qquad \qquad \qquad $\qed$

\section{Weak {\LP modules} over dg Lie algebras}\label{Sec:LPmodule}


\subsection{Loday-Pirashvili modules and dg Leibniz algebras}\label{Sec:adjunction}

In~\cite{LP}, Loday and Pirashvili showed that there exists a pair of adjoint functors between the category of  Leibniz algebras and that of Lie algebra objects in the category $\mathcal{LM}$ of linear maps.
It is also proved \emph{loc.cit.} that each Lie algebra object in the category $\mathcal{LM}$ is equivalent to a morphism of $\g$-modules to the adjoint module for some ordinary Lie algebra $\g$,  we  call such a morphism of $\g$-modules a \textbf{\LP module}. We first generalize this notion in  the dg Lie algebra setting as follows:
\begin{Def}
A \textbf{\LP module} consists of a dg Lie algebra $\g$, a $\g$-module $M$, and a morphism $f \colon M \to \g$ of $\g$-modules.
A morphism $\Phi = (\underline{\phi},\phi)$ of   {\LP modules} from $f \colon M \to \g$ to $f^\prime \colon {M^\prime} \to \g^\prime$ is specified by the following commutative diagram in ${\mathbf{Ch}}(\k)$:
\[
    \begin{tikzcd}
    M \ar{d}[left]{f} \ar{r}{\phi} & {M^\prime} \ar{d}{f^\prime} \\
   \g  \ar{r}{\underline{\phi}} & \g^\prime,
  \end{tikzcd}
\]
such that
\begin{itemize}
  \item $\underline{\phi} \colon \g \to \g^\prime$ is a morphism of dg Lie algebras;
  \item $\phi \colon M \to {M^\prime}$ is a morphism of $\g$-modules covering $\underline{\phi}$, that is,
  \[
     \phi(\rho_{\g,M}(x; m)) = \rho_{\g^\prime,{M^\prime}}(\underline{\phi}(x); \phi(m)),
  \]
  for all $x \in \g, m \in M$.
\end{itemize}
We denote by ${\LPM}$ the category of  \LP modules.
\end{Def}
Given a dg Lie algebra $\g$, the category ${{\LPM }}$ admits a subcategory ${\LPM _{\g}}$ consisting of  {\LP modules} over $\g$, together with morphisms specified by commutative triangles in the category $\Module_{\g}$:
\begin{equation*}\label{Eq: CT in  Repg}
 \begin{tikzcd}
      M \ar{rr}{\phi}  \ar{dr}[swap]{f} && {M^\prime} \ar{dl}{f^\prime} \\
       &  \g. &
  \end{tikzcd}
\end{equation*}

  In the language of functor of points, the category $\LPM_{\g}$ is formed by  points of adjoint modules of $\g$ (``adjoint-points" for short) in the category $\Module_\g$ of $\g$-modules.

\begin{Def}\label{Def:dgLeibnizoverk}
A \textbf{dg Leibniz algebra} $L$ over a dg algebra $\A$ is a  Leibniz algebra in the category of   $\A$-modules, i.e., an $\A$-module $ (L,d)$ together with a morphism of   $\A$-modules
$\diamond\colon  L \otimes L \rightarrow L$, called the Leibniz bracket, which is subject to the (left) Leibniz rule:
\[
 x \diamond (y\diamond z) = (x \diamond y) \diamond z  + (-1)^{\abs{x}\abs{y}} y \diamond (x \diamond z), \quad \forall~ x,y,z \in L.
\]
A morphism of dg Leibniz algebras over $\A$ from $L = (L,d, \diamond)$ to $L^\prime = (L^{\prime }, d^\prime, \diamond^\prime)$ is a morphism of dg $\A$-modules $f \colon (L, d) \to (L^{\prime },d^\prime)$ such that $f(x \diamond y) =  f(x) \diamond^\prime f(y)$ for all $x,y \in L$.

We denote by $\dgLeib_{\A}$ the category of dg Leibniz algebras over $\A$.
In particular, when the dg algebra $\A$ is the base field $\k$ with zero differential, we obtain the category $\dgLeib_{\k}$ of dg Leibniz  algebras over $\k$. For simplicity, $\dgLeib_{\k}$ is written as  $\dgLeib$.
\end{Def}

\begin{prop}\label{Prop:LPtodgLeib}
Given a  {\LP module} $f \colon M \rightarrow \g$, there exists a dg Leibniz algebra structure (over $\k$) on the cochain complex $M=(M, d^M)$ defined by
\begin{equation}\label{Eq: Leibniz bracket from f}
	m_1 \diamond m_2 = \rho_{\g,M}(f(m_1); m_2)\quad (=\iota_{f(m_1)[1]} d_{1}^M m_2),
\end{equation}
for all $m_1, m_2 \in M$.
A morphism $\Phi = (\phi, \underline{\phi})$ of  {\LP modules} from $f \colon M \to \g$ to $f^\prime \colon {M^\prime} \to \g^\prime$ induces a morphism $\phi \colon M \to {M^\prime}$ of dg Leibniz algebras.
\end{prop} This proposition is clearly a generalization of the original construction of Leibniz algebras in \cite{LP}. The proof is omitted as it is a quite straightforward verification.
\begin{Rem}\label{Rem:LPgradedcase} If $\g$ is a graded Lie algebra (i.e., a dg Lie algebra whose differential is zero) and $M$ is a graded $\g$-module (i.e.,  the cochain complex $M$ with zero differential admits a  $\g$-action), then Proposition \ref{Prop:LPtodgLeib} gives a graded Leibniz algebra structure $\diamond$ on $M$.  This construction is the original   method of \LP in \cite{LP}.
\end{Rem}

 Based on Proposition \ref{Prop:LPtodgLeib}, we obtain a functor
\begin{equation}\label{Functor:mathcalG}
\mathcal{G} \colon {{\LPM }} \rightarrow \dgLeib.
\end{equation}
Conversely, given a dg Leibniz algebra $L = (L, d, \circ)$, there is a dg Lie algebra defined by
\[
 L_{\operatorname{Lie}}:= L/{K},
\]
where $K$, called the Leibniz kernel of $L$, is the two-sided ideal generated by elements $x \circ y + (-1)^{\abs{x}\abs{y}}y \circ x$ for all homogeneous $x,y \in L$. One can examine that  $L$ is an $L_{\operatorname{Lie}}$-module and the projection $\pr\colon L \to L_{\operatorname{Lie}}$ is an object in ${{\LPM }}$.
Meanwhile, given a morphism $\phi \colon L \to L^\prime$ of dg Leibniz algebras, since $\phi(K) \subset  K^\prime$, it follows that $\phi$ induces a morphism of dg Lie algebras $\underline{\phi} \colon L_{\operatorname{Lie}} \to L^\prime_{\operatorname{Lie}}$. The commutative diagram
\[
\begin{tikzcd}
    L \ar{d}[left]{\pr} \ar{r}{\phi} & L^\prime \ar{d}{\pr^\prime} \\
   L_{\operatorname{Lie}}  \ar{r}{\underline{\phi}} & L^\prime_{\operatorname{Lie}},
  \end{tikzcd}
\]
gives rises to a morphism $\Phi = (\underline{\phi}, \phi)$ between the associated  {\LP modules}. These facts boil  down to a functor
\[
\mathcal{F} \colon  \dgLeib \rightarrow \LPM .
\]
Moreover,  $\mathcal{G}$ and $\mathcal{F}$  make an adjunction of functors:
\begin{equation}\label{Eq: Rep to Leib functor}
 \mathcal{F} \colon \dgLeib  \rightleftharpoons {{\LPM }} \colon \mathcal{G}.
\end{equation}
Indeed, for any object $L \in \dgLeib $, we have $\mathcal{G}(\mathcal{F}(L)) = L$; and for any object $f \colon M \to \g$ in $\LPM$, we have
\[
   (\mathcal{F}\mathcal{G})(f \colon M \to \g)) = \pr \colon M \to M_{\operatorname{Lie}}.
\]
In addition, the adjunction morphism
\[
\sigma(f \colon M \to \g) \colon (\mathcal{F}\mathcal{G})(f \colon M
\to \g) \to (f \colon M \to \g)
\]
 is specified by the following commutative diagram:
 \[
   \begin{tikzcd}
    M \ar{d}[left]{(\mathcal{F}\mathcal{G})(f \colon M
    	\to \g)=\pr} \ar{r}{\id} & M \ar{d}{f} \\
   M_{\operatorname{Lie}}  \ar{r}{\underline{\phi}} & \g.
  \end{tikzcd}
 \]
The existence of $\underline{\phi}$ stems from the fact that the Leibniz kernel of the dg Leibniz algebra $\mathcal{G}(f \colon M \to \g)$ lies in the kernel of $f$.

 When we  concern with the  restriction of  $\mathcal{\G}$ on ${\LPM _{\g}}$, we get
\begin{equation}\label{Eq: Restriction of G to g}
 \mathcal{G}_\g \colon  {\LPM _{\g}}  \to \dgLeib ,\qquad (f\colon M \to \g)\mapsto (M,\diamond).
\end{equation}
Since the category ${\LPM _{\g}}$ consists of adjoint-points, i.e., morphisms $\Module_{\g}(-,\g)$  in the category $\Module_\g$ of $\g$-modules,  the Chevalley-Eilenberg functor $C(\g,-)$ sends them to the category  of $C(\g,\g)$-points which we denote by $ (\Mod_{C(\g)})_{ C(\g,\g)}$, i.e., morphisms $\Mod_{C(\g)}\left(-, C(\g,\g)\right)$  in the category $\Mod_{C(\g)}$ of   $C(\g)$-modules.
Hence, the Chevalley-Eilenberg functor $C(\g,-)$ induces a functor  which we write again as
\begin{equation*}\label{Eqt:Cg-induced}
C(\g,-)\colon {\LPM _{\g}} \to (\Mod_{C(\g)})_{ C(\g,\g)}.
\end{equation*}
Combining with the restricted functor $\mathcal{G}_\g$ in~\eqref{Eq: Restriction of G to g}, we obtain the following functor
\begin{align*}
\mathcal{G}_{C(\g)} \colon {\LPM _{\g}}  &\to \dgLeib_{C(\g)}, \\
(f \colon M \to \g) &\mapsto (C(\g,M), d_{\tot}^M, \diamond),
\end{align*}
where the Leibniz bracket $\diamond$ arises from  the one in~\eqref{Eq: Leibniz bracket from f} by $C(\g)$-bilinear extensions. The definition of $\mathcal{G}_{C(\g)}$ on morphisms in ${\LPM _{\g}}$ is apparent.

It is natural to ask if all   dg Leibniz algebras $(C(\g,M), d_{\tot}^M, \diamond)$ over $C(\g)$ arise  from those $C(\g,\g)$-points in the category $\Mod_{C(\g)}$, i.e., if there exists a functor from $(\Mod_{C(\g)})_{ C(\g,\g)} $ to $\dgLeib_{C(\g)}$ making the following diagram commute
\[
 \begin{tikzcd}
               & (\Mod_{C(\g)})_{ C(\g,\g)}  \ar[d, dashed, "\exists ?"] \\
    {\LPM _{\g}}  \ar{ur}{C(\g, -)} \ar{r}[swap]{\mathcal{G}_{C(\g)}} & \dgLeib_{C(\g)}.
  \end{tikzcd}
\]
The answer is negative. The best that we can say is that given any object in $(\Mod_{C(\g)})_{ C(\g,\g)} $, i.e., a morphism  $f = \sum_{p \geqslant 0} f_p \in \Mod_{C(\g)}(C(\g,M), C(\g,\g)) $ where $f_p$ is the weight $p$ component,   by a generalization of Equation~\eqref{Eq: Leibniz bracket from f} (see Theorem \ref{Thm:wLP to HLA}), the morphism $f$ produces a Leibniz$_\infty$ algebra structure on $C(\g, M)$. In other words,  the correct answer is the larger   commutative diagram below:
\[
\begin{tikzcd}
& (\Mod_{C(\g)})_{ C(\g,\g)}  \ar[dr, dashed, " "] &\\
{\LPM _{\g}}  \ar{ur}{C(\g, -)} \ar{r}[swap]{\mathcal{G}_{C(\g)}} &\dgLeib_{C(\g)} \ar{r}[swap]{ \mathrm{inclusion}}
& \Leib^\infty_{C(\g)}.
\end{tikzcd}
\]
In fact, objects in the category $(\Mod_{C(\g)})_{ C(\g,\g)}$ corresponds to ``adjoint-points" in the larger category  $\lfMod_\g^\infty$  of locally finite $\infty$-$\g$-modules. Such ``adjoint-points" are called weak {\LP modules} over $\g$, which is explained in the subsequent section.


\subsection{Weak {\LP modules}}\label{Sec:WeakLPmodules}
\begin{Def}\label{Def: wLP}
 A \textbf{weak {\LP module}} over a dg Lie algebra $\g$ is a {\LF} $\infty$-$\g$-module $M = (M,d^M_{\tot})$ together with {an $\infty$-morphism} $f \colon M \rightsquigarrow \g$ from $M$ to the adjoint module of $\g$, i.e., a collection of linear maps
  \[
  f_k \colon     M \to C^k(\g,\g) ,\quad k \geqslant 0,
  \]
  subject to the following constraints:
  \begin{compactenum}
  \item (Local finiteness) For  all $m \in M$, there exists a sufficiently large integer $N_m$  such that $f_k(  m) = 0$ for all $k \geqslant N_m$;
  \item The map $F=\id_{C(\g)}\otimes \sum_{k=0}^\infty f_k$  $\colon C(\g,M) \to C(\g,\g)$ intertwines the relevant differentials, i.e.,
      \[
      F \circ d_{\tot}^{M }=d_{\tot}^{\g }\circ F.
      \]
  Here $d_{\tot}^{M }=d_{\CE} +\sum_{k=0}^\infty d_k^M$ is the total    differential  on $M$ as in~\eqref{Eqt:dCEVoriginalinfinity}, and $d_{\tot}^{\g }$ is the one in Equation \eqref{Eqt:dtotg}.

\end{compactenum}

\end{Def}
We denote such a weak {\LP module} over $\g$ by $(f \colon M \rightsquigarrow \g)$.
Passing to the tangent cohomology, it is easy to see that  $\tanH(f_0) \colon \tanH^\bullet(M) \to \tanH^\bullet(\g)$ is a Loday-Pirashvili module, i.e., a morphism of $\tanH^\bullet(\g)$-modules.

The notion of a dg \LP module introduced in \cite{Qiaoold},   is a special kind of  weak \LP modules.

\begin{Ex}\label{Ex: pairing}
Let  $\alpha \in Z^2(\g)$ be a Chevalley-Eilenberg $2$-cocycle of total degree $k$, in other words, $\alpha \in (S^2(\g^\vee[-1]))^k$   is $d_{\CE}$-closed.
Consider the tensor product $M:= \g[1-k] \otimes \g$ of the adjoint module $\g$ with a shifted copy $\g[1-k]$. Define a linear map
\begin{align*}
	f_1  &\colon  \g[1-k] \otimes \g \to C^1(\g) \otimes \g , &
	f_1 ( y[1-k] \otimes z) &:= \iota_{y[1]}\alpha  \otimes z
\end{align*}
for all $x,y \in \g$. It follows from direct verification that  $f_1$ defines a weak {\LP module} $\g[1-k] \otimes \g \rightsquigarrow \g$.
\end{Ex}

\begin{Ex} \label{Ex: Lie pair}
By a dg Lie algebra pair  $(\largeg,\h)$, we mean that both $\h = (\h, d^{\h}, [-,-]_\h)$ and $\largeg = (\largeg, d^{\largeg}, [-,-]_\largeg)$ are finite dimensional dg Lie algebras and $\h \subset \largeg $ is a dg Lie subalgebra.
Let us denote the quotient complex by $({\bb}, d^{\bb})$. 
 Both cochain complexes $(\largeg, d^{\largeg})$ and $({\bb}, d^{\bb})$  admit $\h$-module structures, and they are defined by
\begin{align*}
	x \triangleright z  &:= [x,z]_\largeg,  & &\mbox{and} &  x \triangleright \pr_{\bb}(z)   &:= \pr_{\bb}([x,z]_\largeg),
\end{align*}
for all $x \in \h, z \in \largeg$, respectively,.
We also have a short exact sequence of $\h$-modules
\begin{equation}\label{SES of VS in Lie pair}
	0 \rightarrow \h \xrightarrow{i} \largeg     \xrightarrow{\pr_{\bb}} {\bb} \rightarrow 0.
\end{equation}
Choose a splitting of Sequence~\eqref{SES of VS in Lie pair} as graded vector spaces, i.e., a pair of linear maps  $j\colon {\bb} \rightarrow \largeg    $ and $\pr_\h\colon \largeg     \rightarrow \h$ satisfying
\begin{align*}
	\pr_\h \circ i &= \id_\h, & \pr_{\bb} \circ j &= \id_{{\bb}}, &\mbox{ and }~ i \circ \pr_\h + j \circ \pr_{\bb} &= \id_\largeg    .
\end{align*}
Via the splitting $(j,\pr_\h)$, we have a decomposition $\largeg=\h\oplus \bb$ of graded vector spaces and a decomposition $\largeg^\vee=\h^\vee\oplus \bb^\vee$ of dual spaces as well.

The failure of $j$ being a cochain map determines a linear map
\[
 f_0 \colon \bb[-1] \to \h,
\]
which is defined for all homogeneous element $b[-1] \in \bb[-1]$ by
\[
 f_0(b[-1]) = (-1)^{\abs{b}}d(j)(b):= (-1)^{\abs{b}}\left(d^{\largeg}(j(b)) - j(d^{\bb}(b))\right).
\]
It is clear that $f_0$ satisfies
\[
d^{\h}\circ f_0 - f_0 \circ d^{\bb} = 0 \colon \bb[-1] \to \h[1].
\]
The $\h$-module structure on $\largeg$ can be described by
\[
d^{\largeg}_1 l= \pr_{\h^\vee[-1]\otimes \largeg}( \dadjoint l) = (i^\ast \otimes\id_{\largeg})( \dadjoint l),\quad \forall~ l\in \largeg,
\]
where $ \dadjoint\colon \largeg\to \largeg^\vee[-1]\otimes \largeg$ stands for the adjoint $\largeg$-module structure of $\largeg$.
Similarly, the $\h$-module structure on $\bb$ can be described by
\[
d^{\bb}_{1} b=  \pr_{\h^\vee[-1]\otimes \bb}(\dadjoint b)=(i^\ast\otimes\pr_{\bb})(\dadjoint b) ,\quad \forall~ b\in \bb.
\]
Consider the linear map
\[
 f_1 \colon \bb[-1] \to C^1(\g,\h) = \g^\vee[-1] \otimes \h,
\]
which is defined for all $b[-1] \in {\bb[-1]}$ by
\[
f_1(b[-1]) =    (-1)^{\abs{b}}\pr_{\h^\vee[-1]\otimes \h}(\dadjoint b)=(-1)^{\abs{b}} (i^\ast \otimes\pr_{\h})(\dadjoint b).
\]
It follows from a direct verification that $f = f_0 + f_1$ yields a weak {\LP module} $ {\bb}[-1]\rightsquigarrow  \h$.
\end{Ex}
\begin{Rem}\label{Rem: homotopy question on two examples}
It is natural to ask when we choose two cohomologous cocycles $\alpha$ and $\alpha^\prime \in Z^2(\g)$ in Example \ref{Ex: pairing}, what is the relationship between the  resulting  weak {\LP modules} $f$ and $f^\prime$;
We also wish to see how the weak {\LP module} in Example~\ref{Ex: Lie pair} depends on the choice of splitting $j$ of Sequence~\eqref{SES of VS in Lie pair}. These questions are dealt with in Examples \ref{Ex: pairing continued} and \ref{Ex: Lie pair continued} of the subsequent section, where we construct a homotopy theory for weak {\LP modules}.
\end{Rem}

\begin{Def}\label{Def: morphisms of hets}
A \textbf{morphism} of weak {\LP modules} over $\g$ from $(f \colon M \rightsquigarrow \g)$ to $(g \colon N \rightsquigarrow \g)$ is {an $\infty$-morphism} of {\LF} $\infty$-$\g$-modules $\phi \colon M \rightsquigarrow N$ such that the following  triangle is commutative in the category $\lfMod_\g^\infty$:
\[
 \begin{tikzcd}
       M \ar[rr, rightsquigarrow, "\phi"]  \ar[dr, rightsquigarrow,swap, "f"] && N \ar[dl, rightsquigarrow, "g"] \\
       &  \g. &
  \end{tikzcd}
\]
\end{Def}
Compositions of morphisms of weak {\LP modules} over $\g$ are defined in an obvious manner.
With the descriptions provided in Definitions \ref{Def: wLP} and \ref{Def: morphisms of hets}, it follows that weak {\LP modules} over $\g$ are naturally gathered in a category which we denote by ${\weakLPmodg}$. In other words, it is the category consisting of ``adjoint-points" in the category $\lfMod_\g^\infty$ of locally finite $\infty$-$\g$-modules
\[
\weakLPmodg=\{\lfMod_\g^\infty(M,\g) \mid M\in \lfMod_\g^\infty\}.
\]
One can also consider the category $\mathbf{wLP} $ of
weak {\LP modules} without specifying the dg Lie algebra.
We defer the discussion of related content to Section \ref{Sec:Functoriality}.

Applying Proposition~\ref{prop: CE functor from infty modules}, we immediately obtain the following fact.
\begin{prop}\label{Prop:Cinducesequivalence}
The Chevalley-Eilenberg functor $\secondC(\g,-) \colon \lfMod_\g^{\infty} \to \Mod_{C(\g)}$ induces an equivalence of categories
\[
 \secondC(\g, -) \colon {\weakLPmodg} \xrightarrow{~\simeq~} \left(\fgsMod_{C(\g)}\right)_{ C(\g,\g)} .
\]
Here $\left(\fgsMod_{C(\g)} \right)_{ C(\g,\g)} $ stands for the category of finitely generated and  semifree   $C(\g)$-modules over $ C(\g,\g)$.
\end{prop}

\subsection{Weak {\LP modules} in terms of dg geometry}\label{Sec:dgpoint}
From the dg geometry point of view, a dg Lie algebra $\g$ is a (formal) pointed dg manifold $B\g := (\g[1], d_{\CE})$, whose space of functions $C^\infty(B\g)$ is naturally identified with the Chevalley-Eilenberg dg algebra $C(\g)$.
Now we give an equivalent characterization of weak {\LP modules} over $\g$ (via formal geometry) in terms of the dg manifold $B\g$.

Denote the category of dg vector bundles over the dg manifold $B\g$ by $\mathbf{dgVec}_{B\g}$. According to Mehta, Sti\'{e}non, and Xu~\cite{MSX}*{Lemma 1.5}, this category is equivalent to the category $\fgsMod_{C(\g)}$ of finitely generated semifree   $C(\g)$-modules.
Consider the essentially surjective and fully faithful functor obtained by taking global sections
\[
 \Gamma \colon \mathbf{dgVec}_{B\g} \xrightarrow{\simeq}  \fgsMod_{C(\g)},
\]
i.e., given a dg vector bundle $\mathcal{E} \to B\g$, its space $\Gamma(\mathcal{E})$ of global sections is a semifree  $C(\g)$-module.
Meanwhile, by Proposition~\ref{prop: CE functor from infty modules}, $\secondC(\g,-) \colon \lfMod_\g^{\infty} \xrightarrow{\simeq} \fgsMod_{C(\g)}$ is an equivalence of categories. Therefore, we have an equivalence of categories
\[
   \lfMod_\g^{\infty} \simeq \mathbf{dgVec}_{B\g}.
\]
Thus, given a {\LF} $\infty$-$\g$-module $M$, there exists a dg vector bundle $\mathcal{E} \to B\g$, whose space of global sections is naturally identified with the semifree   $C(\g)$-module $\secondC(\g,M)$.
Moreover, since the global section space $\Gamma(\TminusoneBg )$ of the shifted tangent bundle $\TminusoneBg $ $\to$ $B\g$ can be naturally identified with the  $C(\g)$-module $C(\g,\g)$, combining with Proposition~\ref{Prop:Cinducesequivalence}, we immediately obtain the following fact.
\begin{prop}\label{prop-def via dg bundle}
  Let $\g$ be a dg Lie algebra. There is a one-to-one correspondence between weak {\LP modules}  over $\g$ and morphisms in the category $\mathbf{dgVec}_{B\g}$  targeting at the shifted tangent bundle $\TminusoneBg $.
\end{prop}

Dually, we have a characterization of weak {\LP modules} over $\g$ by the cotangent complex of   $B\g$.
\begin{Def}
 Let $\mathcal{U}\in \Mod_{C(\g)}$ be a     $C(\g)$-module. A $\mathcal{U}$-valued dg derivation of $C(\g)$ is a $\k$-linear cochain map $\delta:  C(\g) \rightarrow \mathcal{U}$ satisfying the Leibniz rule
 \[
  \delta(\xi \odot \eta) = (-1)^{\abs{\xi}\abs{\eta}}\eta \odot \delta(\xi) + \xi \odot \delta(\eta),\;\;\forall~ \xi,\eta \in C(\g).
 \]
\end{Def}
\begin{Rem}
The trivial square zero extension of $C(\g)$ by $\mathcal{U}$ is an augmented $C(\g)$-algebra fitting into  the following natural morphisms:
  \begin{align*}
   C(\g) \xrightarrow{i} &C(\g) \oplus \mathcal{U} \xrightarrow{\pr_1} C(\g).
  \end{align*}
Each section of the projection $\pr_1: C(\g) \oplus \mathcal{U} \to C(\g)$ gives rise to a $\mathcal{U}$-valued dg derivation of $C(\g)$, and vice versa.
\end{Rem}
Let $\Der(C(\g),\mathcal{U})$ be the set of $\mathcal{U}$-valued dg derivations of $C(\g)$, and it admits an obvious   $C(\g)$-module structure.  Meanwhile, each morphism $\Phi\colon \mathcal{U} \to \mathcal{M}$ of   $C(\g)$-modules induces a morphism $\Phi \circ - \colon \Der (C(\g),\mathcal{U}) \to \Der(C(\g),\mathcal{M})$ of dg derivations of $C(\g)$ defined by compositions. Hence, we obtain a functor $\Der(C(\g),-)$ from the category  $\Mod_{C(\g)}$ of   $C(\g)$-modules to itself.
It is well-known that this functor is corepresented by the  $C(\g)$-module $\Omega_{C(\g)\mid\k}$ of K\"{a}hler differentials in the following sense:
 \[
 \Der(C(\g),-) \cong \Mod_{C(\g)}(\Omega_{C(\g)\mid\k},-) \colon ~ \Mod_{C(\g)} \to \Mod_{C(\g)}.
 \]
For this reason, $\Omega_{C(\g)\mid\k}$ is also called the cotangent complex of the dg manifold $B\g$.
	
Given a {\LF} $\infty$-$\g$-module $M$, let $ M^{\vee}[-1]\in \lfMod_\g^{\infty} $ be the dual object with shifted degrees. There exists a one-to-one correspondence between weak {\LP modules} $f \colon M \rightsquigarrow \g$ and $\secondC(\g,M^{\vee}[-1])$-valued dg derivations $\delta$ of $C(\g)$.
In fact, if $f = \sum_k f_k$, where $f_k \colon M \to C^k(\g,\g)$, is a weak Loday-Pirashvili module over $\g$, then its linear dual map $f^{\vee} = \sum_k (f_k^{\vee} \colon \g^{\vee}[-1] \to C^k(\g, M^{\vee}[-1]))$ determines a morphism $F^{\vee} \colon \Omega_{C(\g)\mid\k} \to C(\g, M^{\vee}[-1])$ of   $C(\g)$-modules, and gives rise to a $C(\g,M^{\vee}[-1])$-valued dg derivation $\delta$ of $C(\g)$.
As a consequence, we obtain the following characterization of weak {\LP modules} over $\g$.
\begin{prop}\label{prop: fin het and der}
The category ${\weakLPmodg}$ of  weak {\LP modules} over $\g$ is equivalent to the category $\Der(C(\g),\fgsMod_{C(\g)})$   of dg derivations of $C(\g)$ valued in finitely generated and semifree dg $C(\g)$-modules.
\end{prop}

\subsection{A homotopy theory for weak {\LP modules}}\label{Section: homotopy theory}

Since weak {\LP modules} are defined to be ``adjoint-points" in the category $\lfMod_\g^\infty$,
they inherit a homotopy equivalence relation as well.
\begin{Def}\label{Def: homotopy of hets}
  Two weak {\LP modules} $f, f^\prime \colon M \rightsquigarrow \g$ over a dg Lie algebra $\g$ are said to be \textbf{homotopic} if they are homotopic as $\infty$-morphisms in the category $\lfMod_\g^\infty$, i.e., there exists {an $\infty$-morphism} $H \colon M \rightsquigarrow J(\g)$ which we call  a homotopy from $f$ to $f^\prime$, such that the following diagram in   $\lfMod_\g^{\infty}$ commutes:
  \[
  \begin{tikzcd}
         & M \ar[rightsquigarrow]{ld}[left]{\overset{\mathlarger{f}}{\overset{~ \overset{~ }{~ }}{~\overset{~ }{~ } }}} \ar[d, rightsquigarrow, "H"] \ar[dr, rightsquigarrow, "f^\prime"] & \\
  \g   & J(\g) \ar{l}{\epsilon_0} \ar{r}[swap]{\epsilon_1} & \g.
  \end{tikzcd}
  \]
\end{Def}

For each weak {\LP module} $(f \colon M \rightsquigarrow \g)$ $\in{\weakLPmodg}$, we will denote its homotopy class by $[f  \colon M \rightsquigarrow \g]$, or simply $[f]$.
Morphisms between homotopy classes of weak {\LP modules} are defined in a natural fashion:
Given two weak {\LP module} classes $[f]=[f\colon M \rightsquigarrow \g]$ and $[g]=[g  \colon N \rightsquigarrow \g]$, a morphism $\phi$ from $[f]$ to $[g]$ is {an $\infty$-morphism} $\phi \colon M \rightsquigarrow N$ such that the triangle in Definition~\ref{Def: morphisms of hets} commutes in the homotopy category $\Ho(\lfMod_\g^\infty) (\simeq ~\lfMod_\g^\infty{/_\sim})$, i.e., $[f] = [g \circ \phi]$.
So we can  form the category $\Ho(\weakLPmodg)$ of  \textit{homotopy classes} of weak {\LP modules} over $\g$.

Similarly, we can define the homotopy category
$\Ho(\Der(C(\g),\fgsMod_{C(\g)}))$ whose objects are \textit{homotopy classes} of dg derivations of $C(\g)$ valued in finitely generated and semifree   $C(\g)$-modules.  From Propositions~\ref{prop: fin het and der} and~\ref{prop: homotopy invariance of CE functor}, we get an important fact.

\begin{prop}\label{prop: equivalence of weak ad and derivations}
  There is an equivalence of categories
  \[
  \Ho(\weakLPmodg) \simeq \Ho(\Der(C(\g),\fgsMod_{C(\g)})) .
   \]
  \end{prop}

At the end of this section, we unfold  Definition~\ref{Def: homotopy of hets} to describe explicitly homotopy equivalence of weak {\LP modules}.
\begin{prop}
  Two weak {\LP modules} $f, f^\prime \colon M \rightsquigarrow \g$ are homotopic if and only if there exists a sequence of  linear maps
  \[
  h_n \colon M \to C^{n}(\g, \g)[-1], \quad n = 0,1,2,\cdots,
  \]
  which is locally finite (similar to (1) of Definition \ref{Def: wLP}) and
  such that the associated $C(\g)$-linear extensions $F$ of $f$ and $F^\prime$ of $f^\prime$  satisfy
  \[
    F - F^\prime =  d_{\tot}^\g \circ h + h \circ d_{\tot}^M \colon~~~ C(\g,M) \to C(\g,\g),
  \]
  where $h = \sum_n h_n \colon C(\g,M) \to C(\g,\g)[-1]$ is $C(\g)$-linear.
We call the sequence of maps $\{h_n\}$ an $\infty$-homotopy between maps $f$ and $f^\prime$.
\end{prop}
\begin{proof}
   Note that the path object for the adjoint module of $\g$, namely
   \[
   J(\g) = (\g \oplus \g t \oplus \g[-1] dt, d_{\tot}^{J(\g)} = d_0^{J(\g)} + d_1^{J(\g)}),
   \]
   is defined by
  \begin{align*}
    d_0^{J(\g)}(x+ty+(dt)z) &= d_0^\g(x) + td_0^\g(y)  + dt(y-d_0^\g(z)), \\
    d_1^{J(\g)}(x + ty+(dt)z) &= d_1^\g(x) + td_1^\g(y)  + dt(d_1^\g(z)),
  \end{align*}
  for all $x+ty+(dt)z \in \g \oplus \g t \oplus \g[-1] dt$.

  Assume that $f$ and $f^\prime$ are homotopic via the right homotopy $H \colon M \rightsquigarrow J(\g)$. We show how to find the maps $h_n$.

  By definition, $H$ is a $C(\g)$-linear map
  \[
   H= \sum_n H_n \colon  M \to \oplus_n C^n(\g, J(\g)),
  \]
  satisfying
  \[
    H \circ d_{\tot}^M = d_{\tot}^{J(\g)} \circ H \colon C(\g,M) \to C(\g, J(\g)).
  \] Since $J(\g) = \g \oplus \g t + \g[-1] dt$, we may write $H_n = (H_n^{(1)}, H_n^{(2)}, H_n^{(3)})$, where
  \begin{align*}
    H_n^{(1)} &\colon  M \to C^n(\g,\g), \\
    H_n^{(2)} &\colon  M \to C^n(\g,\g)t, \\
  \mbox{and}\quad  H_n^{(3)} &\colon M \to C^n(\g,\g)[-1]dt.
  \end{align*}
  Since $\epsilon_0 \circ H = f$ and $\epsilon_1 \circ H = f^\prime$, it follows that
  \begin{align*}
    H_n^{(1)} = f, &\quad \mbox{and} \quad H_n^{(2)}  = f^\prime - f.
  \end{align*}
  Accordingly, we can define a sequence of multilinear maps $h_n \colon M \to C^n(\g,\g)[-1]$  (for all $n \geqslant 0$)  by
  \[
  H_n^{(3)}(m) = h_n(m)dt.
  \]
  The resulting $\{h_n\}_{n\geqslant 0}$ is subject to the constraint stated by the proposition.   The converse argument is  achieved similarly, and thus omitted.
\end{proof}

\begin{Ex}\label{Ex: pairing continued}
By Example~\ref{Ex: pairing}, each $2$-cocycle $\alpha \in Z^2(\g)$ of degree $k$ defines a weak {\LP module} $f_1 \colon M = \g[1-k] \otimes \g \rightsquigarrow \g$. We claim that cohomologous $2$-cocycles of degree $k$, say $\alpha$ and $\alpha^\prime \in Z^2(\g)$,  induce homotopic weak {\LP modules}.
	
In fact, by assumption   there exists $\eta \in C(\g)$ of total degree $k-1$ such that $\alpha-\alpha^\prime= d_{\CE}\eta$. Now, the weak {\LP module} arising from $\alpha^\prime$ is given by
\begin{align*}
	f_1^\prime  &\colon  \g[1-k] \otimes \g \to C(\g)\otimes \g , &
	f_1^\prime ( y[1-k] \otimes z) &:= \iota_{y[1]}\alpha^\prime  \otimes z.
\end{align*}
We can derive that
  \[
	F_1 - F_1^\prime=\id_{C(\g)}\otimes(f_1 - f_1^\prime) =  d_{\tot}^\g \circ h + h \circ d_{\tot}^M \colon C(\g,\g[-1] \otimes \g) \to C(\g,\g),
\]
where $h   \colon C(\g, \g[1-k] \otimes \g) \to C(\g,\g)[-1]$ is the $C(\g)$-linear extension of
\[
h\colon  \g[1-k] \otimes \g  \to C(\g,\g)[-1],\qquad\qquad h(y[1-k] \otimes z):=(-1)^{\abs{y}}\iota_{y[1]}\eta  \otimes z.
\]
This proves that $h$ gives rise to an $\infty$-homotopy from $f_1$ to $f_1^\prime$.
\end{Ex}
\begin{Ex} \label{Ex: Lie pair continued}
Let $(\largeg, \h)$ be a pair of dg Lie algebras as in Example~\ref{Ex: Lie pair}. For each  splitting of Sequence~\eqref{SES of VS in Lie pair}, we obtain a weak {\LP module} $f = f_0 + f_1\colon {\bb}[-1]\rightsquigarrow \h$.
A different splitting of Sequence~\eqref{SES of VS in Lie pair} is determined by a linear map $\lambda\colon \bb\to \h$ such that the new splitting is given by
\[
j^\prime=\lambda+j\colon {\bb} \rightarrow \largeg,  b\mapsto \lambda(b)+j(b), \quad \mbox{and}\quad \pr^\prime_\h = \pr_\h-\lambda \colon \largeg  \rightarrow \h,  x+j(b)\mapsto x-\lambda(b).
\]
From this splitting one obtains another weak {\LP module} $f^\prime = f_0^\prime +  f^\prime_1\colon {\bb}[-1]\rightsquigarrow \h$, where the corresponding maps $f_0^\prime \colon \B[-1] \to \h$ and $f_1^\prime \colon B[-1] \to C^1(\g, \g)$ are given by
\[
 f_0^\prime(b) = (-1)^{\abs{b}}d(j^\prime)(b) = (-1)^{\abs{b}}d(j+\lambda)(b) = f_0(b) + (-1)^{\abs{b}}d(\lambda)(b),
\]for all homogeneous $b\in \bb$,
and
\begin{eqnarray*}
	f_1^\prime(b) &:=& (-1)^{\abs{b}}(i^\ast \otimes\pr^\prime_{\h})(\dadjoint j^\prime( b))  \\
      &=&  (-1)^{\abs{b}}(i^\ast\otimes(\pr_{\h}-\lambda))(\dadjoint(\lambda(b)+b))= f_1(b) + (-1)^{\abs{b}} \dadjoint \circ\lambda(b) -(i^\ast \otimes\lambda) (\dadjoint b).
\end{eqnarray*}
We now show that $f$ and $f^\prime$ are homotopic.	
In fact, we can define $h\colon \bb[-1] \to \h[-1]$ by
\[
h(b[-1]) := \lambda(b)[-1],\quad \forall~ b\in \bb.
\]
By $C(\h)$-linear extension, we regard $h$ as a map $C(\h,\bb[-1]) \to C(\h,\h)[-1]$. It is clear that
\[
	F - F^\prime =  d_{\tot}^\h \circ h + h \circ d_{\tot}^\bb \colon C(\h,\bb) \to C(\h,\h).
\]
Thus, $h$ builds an $\infty$-homotopy between $f$ and $f^\prime$.
\end{Ex}

\section{From weak {\LP modules} to Leibniz infinity algebras}\label{Section: HLA}
In this section, we elaborate a natural construction of Leibniz$_\infty$ algebra from a weak {\LP module}.
Let $f \colon M \rightsquigarrow \g$ be a weak {\LP module} over a dg Lie algebra $\g$.
Then $\secondC(\g,M)$ is a   $C(\g)$-module. We shall prove that $\secondC(\g,M)$ admits a canonical Leibniz$_\infty$  algebra structure over $C(\g)$. This construction is a natural generalization of the functor $\mathcal{G}$ due to Loday and Pirashvili (see \eqref{Functor:mathcalG}).

\subsection{Leibniz infinity algebras arising from weak {\LP modules}}
We first recall the notion of Leibniz$_\infty$ algebras.
\begin{Def}[\cites{AP, Uchino}]
	A Leibniz$_\infty$ algebra    is a $\Z$-graded vector space $L$ equipped with a collection of multilinear maps $\mu_k\colon \otimes^k L \to L[2-k]$, satisfying
\begin{align*}
	&\sum_{i+j=n+1}\sum_{k=j}^{n}\sum_{\sigma \in \sh(k-j,j-1)}\chi(\sigma)(-1)^{(k+1-j)(j-1)} (-1)^{j(\abs{x_{\sigma(1)}}+\cdots+ \abs{x_{\sigma(k-j)}})} \notag \\
	&\qquad\qquad\qquad\mu_{i}(x_{\sigma(1)},\cdots,x_{\sigma(k-j)},\mu_{j}(x_{\sigma(k-j+1)}, \cdots,x_{\sigma(k-1)}, x_k), x_{k+1},\cdots, x_n)=0,
\end{align*}
	for all $n \geqslant 1$ and all homogeneous elements $x_1,\cdots,x_n \in L$, where $\sh(p,q)$ denotes the set of $(p,q)$-shuffles, and $\chi(\sigma) = \sgn(\sigma)\epsilon(\sigma)$ is the anti-Koszul sign of $\sigma$.
\end{Def}

Suppose that $\A=(\A,d_\A)$ is a dg algebra.
We are particularly interested in Leibniz$_\infty$ algebras which are simultaneous  $\A$-modules.
\begin{Def}\label{Def:LeibnizinfinityA}
A Leibniz$_\infty$ $\A$-algebra is a Leibniz$_\infty$ algebra $(L,\{\mu_k\}_{k\geq1})$   such that the cochain complex $(L,d_1)$ is an $\A$-module, and all higher brackets $\mu_k\colon \otimes^k L \rightarrow L$ ($k\geqslant 2$) are $\A$-multilinear.

A morphism of Leibniz$_\infty$ $\A$-algebras from $(L,\{\mu_k\}_{k\geqslant 1})$ to $(M,\{\nu_k\}_{k\geqslant 1})$ consists of a collection of multi-$\A$-linear maps $\phi_m \colon \otimes^m L \to M[1-m]$ satisfying some natural compatibility conditions (see~\cite{AP}).
\end{Def}
Denote by $\Leib^\infty_{\A}$ the category of Leibniz$_\infty$ $\A$-algebras.
We also need the notion of Leibniz$_\infty[1]$ algebras, which is equivalent to the notion of Leibniz$_\infty$ algebras by a degree shifting.
\begin{Def}[\cites{CSX,CLX}]\label{Def: Leibnizinfty}
	A Leibniz$_\infty[1]$ algebra (over $\k$) is a graded vector space $U = \oplus_{n \in \Z}U^n$,  together with a sequence $\{\lambda_k\colon \otimes^k U \rightarrow U[1]\}_{k \geqslant 1}$ of   $\k$-multilinear maps satisfying
\begin{align}\label{Leib}
&\sum_{i+j=n+1}\sum_{k=j}^{n}\sum_{\sigma \in \sh(k-j,j-1)}\epsilon(\sigma)(-1)^{\abs{x_{\sigma(1)}}+\cdots+ \abs{x_{\sigma(k-j)}}} \notag \\
&\lambda_{i}(x_{\sigma(1)},\cdots,x_{\sigma(k-j)},\lambda_{j}(x_{\sigma(k-j+1)}, \cdots,x_{\sigma(k-1)}, x_k), x_{k+1},\cdots, x_n)=0, \notag
\end{align}
for all $n \geqslant 1$ and all homogeneous elements $x_i \in U$, where $\sh(p,q)$ denotes the set of $(p,q)$-shuffles ($p,q \geqslant 0$), and $\epsilon(\sigma)$ is the Koszul sign of $\sigma$.
	
A morphism of Leibniz$_\infty[1]$ algebras   from $(U,\{\lambda_k\}_{k\geqslant 1})$ to $(K, \{\tau_k\}_{k\geqslant 1})$ consists of a family of multilinear maps $\phi_m: \otimes^m U \to K$ that are compatible with the structure maps $\lambda_k$ and $\tau_k$ (see \cite{CSX}).
\end{Def}
In the same fashion as that of Definition \ref{Def:LeibnizinfinityA}, we can define Leibniz$_\infty[1]$ $\A$-algebras  and their morphisms, and form a category  which we denote  by $\Leib^{\infty[1]}_\A$. Certainly, it is equivalent to the category $\Leib^{\infty}_\A$ by degree shifting.

 The main result of this section is the following theorem.
\begin{Thm}\label{Thm:wLP to HLA}
For each weak \LP module $f \colon M \rightsquigarrow \g$, there exists a Leibniz$_\infty[1]$ $C(\g)$-algebra structure $\{\lambda_k\}_{k\geqslant 1}$ on the $C(\g)$-module $C(\g,M[1])$ whose unary bracket $\lambda_1$ is the Chevalley-Eilenberg differential $d_{\tot}^{M[1]}$.
This correspondence yields a functor
\[
  \mathcal{G}_\g^\infty \colon  \weakLPmodg  \to \Leib^{\infty[1]}_{C(\g)}.
\]
Furthermore, this functor is  homotopy invariant, i.e., if $f^\prime \colon M \rightsquigarrow \g$ is another weak \LP module that is homotopic to $f$, then the two Leibniz$_\infty[1]$ $C(\g)$-algebra structures on $C(\g, M[1])$ are isomorphic to each other.
\end{Thm}
Indeed, this theorem is a direct application of the Kapranov functor that we introduced in
an earlier work \cite{CLX}. It is also a generalization of \cite{Qiaoold}*{Theorem 3.8}.
The proof is a straightforward adaptation of the argument \emph{op.cit.} and is thus omitted.
Below, we give explicitly the construction of $\mathcal{G}_\g^\infty$.

Firstly, we show that for each weak {\LP module} $f \colon M \rightsquigarrow \g$, there is a Leibniz$_\infty[1]$ $C(\g)$-algebra structure $\{\lambda_k\}_{k\geqslant 1}$ on the  $C(\g)$-module $C(\g,M[1])$.
In fact,  given $f\colon M \rightsquigarrow \g$, we have
a dg derivation of $C(\g)$
\[
\delta\colon C(\g) \to \secondC(\g,M^\vee[-1]).
\]
There is a \textbf{$\delta$-connection} on the dg module $C(\g ,M[1])$, i.e.,
 a degree $0$ linear map of graded $\k$-vector spaces
 \[
 \nabla \colon C(\g ,M[1])  \rightarrow  \secondC(\g,M^\vee[-1]\otimes M[1])
 \]
  defined by
   \[
 \nabla(\alpha \otimes m[1]) := \delta(\alpha) \otimes m[1]
 \]
 for all $\alpha \in C(\g)$ and $m  \in M$.
 Alternatively, one can interpret $\nabla$ as the assignment of covariant derivations along elements $b\in C(\g ,M[1])$, i.e., $\nabla_{b} \colon C(\g ,M[1])\to C(\g ,M[1])$. Explicitly, if $b=\xi\otimes u[1] $ where $\xi\in C(\g), u\in M$, then $\nabla_{b}$ is given by
 \[
 \nabla_{\xi\otimes u[1]}(\alpha\otimes m[1])=\xi\odot (\iota_{u[1]}\delta(\alpha))\otimes m[1]=\xi\odot (\iota_{f(u)[1]}\alpha)\otimes m[1].
 \]

According to \cite{CLX} and by applying the Kapranov functor on the dg derivation $\delta$,
 there exists a Leibniz$_\infty[1]$ $C(\g)$-algebra structure $\{\lambda_k\}_{k\geqslant 1}$ on the $C(\g)$-module  $C(\g ,M[1])$:

\begin{itemize}
	\item $\lambda_1 = d_{\tot}^{M[1]}$ is the differential on the  $C(\g)$-module $\secondC(\g,M[1])$;
	\item $\lambda_2$ is $C(\g)$-bilinear  and   determined by
	\begin{align*}\nonumber
 \mathcal{R}^\nabla_2\colon ~M[1]\otimes M[1]&\to C(\g,M[1])[1],  \\ 
	m_1[1]\otimes m_2[1] &\mapsto (-1)^{\abs{m_1}-1} \nabla_{m_1[1]}(d_{\tot}^{M[1]}m_2[1]) =(-1)^{\abs{m_1}-1}\iota_{f(m_1)[1]}\left( \sum_{k=1}^\infty d_k^Mm_2 \right)[1];
\end{align*}
\item $\lambda_n$, $n \geqslant 3$, are all $C(\g)$-multilinear, and are specified by a family of linear maps
\begin{align*}
	\mathcal{R}^\nabla_n\colon (M[1])^{\otimes n} &\rightarrow C(\g,M[1])[1],
\end{align*}
defined  by
\begin{eqnarray}\nonumber
\mathcal{R}^\nabla_{n }(m_1[1],\cdots,m_{n }[1])& =& (-1)^{\abs{m_1}-1}  \nabla_{m_1[1]}\circ \mathcal{R}^\nabla_{n-1}   (m_2[1], \cdots, m_{n }[1])\\\nonumber		&=&(-1)^{\abs{m_1}+\cdots+\abs{m_{n-1}}-n+1} \iota_{f(m_1)[1]}\cdots \iota_{f(m_{n-1})[1]}\left(\sum_{k=1}^\infty d_k^Mm_{n }\right)[1],\\
		&&\label{align:firstlambdan}
\end{eqnarray}
 for all   homogeneous elements $m_1 , \cdots, m_{n }  \in  M $.
 \end{itemize}
In this construction, thanks to the {\LF} condition,  the convergence issue of the infinite sum  $\sum_{k=1}^\infty   d_k^Mm_{n+1}$ does not appear. We also give an explicit description of $\lambda_k$ via a summation over trees in Proposition \ref{Prop: Leibniz infty via OPT}.

In summary,  given an object $(f\colon M \rightsquigarrow \g) \in \weakLPmodg$, we obtain a Leibniz$_\infty[1]$-$C(\g)$-algebra $\mathcal{G}_\g^\infty(f) := ( C(\g ,M[1]),\{\lambda_k\}_{k\geqslant 1})$.
Meanwhile, given a morphism $\phi$ from $(f \colon M \rightsquigarrow \g)$ to $(g \colon N \rightsquigarrow \g)$, there is a morphism $\mathcal{G}_\g^\infty(\phi)$ of Leibniz$_\infty[1]$ $C(\g)$-algebras from $ \mathcal{G}_\g^\infty(f)=( C(\g ,M[1]),\{\lambda_k\}_{k\geqslant 1})$ to $\mathcal{G}_\g^\infty(g)= ( C(\g ,W[1]),\{\nu_k\}_{k\geqslant 1})$ whose first component $\mathcal{G}_\g^\infty(\phi)_1$ is $\phi$. Using the formulae provided by \cite{CLX}*{Proposition 3.9}, all higher components
\[
\mathcal{G}_\g^\infty(\phi)_k \colon C(\g, M[1]) \otimes_{C(\g)} \cdots \otimes_{C(\g)} C(\g,M[1]) \to C(\g,N[1]),\quad  k \geqslant 2
\]
are $C(\g)$-linear, and are given inductively by
\[
 \mathcal{G}_\g^\infty(\phi)_{k+1}(m_1[1],\cdots,m_{k+1}[1]) = \iota_{\phi(m_1[1])} \mathcal{G}_\g^\infty(\phi)_{k} (m_2[1],\cdots,m_{k+1}[1]),\quad  k \geqslant 1.
\]

Finally, we show that the previous assignment is homotopy invariant.
If $f^\prime \colon M \rightsquigarrow \g$ is another weak {\LP module} that is homotopic to $f$, then the associated Leibniz$_\infty[1]$ $C(\g)$-algebra structure $\{\lambda^\prime_k\}_{k\geqslant 1}$ on $C(\g,M[1])$ is isomorphic to the one $\{\lambda_k\}_{k\geqslant 1}$ from $f$.  In fact, by Proposition~\ref{prop: homotopy invariance of CE functor}, their associated $C(\g)$-module morphisms $f,f^\prime \colon \secondC(\g,M[1]) \to \secondC(\g,\g)$ are homotopic. Thus, the associated $\secondC(\g,M^\vee[-1])$-valued dg derivations $\delta$ and $\delta^\prime$ of $f$ and $f^\prime$, respectively, are homotopic. By~\cite{CLX}*{Proposition 3.18}, there exists an isomorphism of Leibniz$_\infty[1]$ $C(\g)$-algebras
\[
\phi_\bullet \colon (C(\g,M[1]),\lambda^\prime_\bullet) \to (C(\g,M[1]), \lambda_\bullet),
\]
whose first map $\phi_1 = \id \colon C(\g,M[1]) \to C(\g,M[1])$.

It is tempting to see what happens when passing to the total Chevalley-Eilenberg cohomology $H_{\tot}(\g,M)$ of the {\LF} $\infty$-$\g$-module $M$.  In general, for a Leibniz$_\infty[1]$ algebra    $U$ whose structure maps are $\lambda_k$ ($k\geqslant 1$), the cohomology $H(U;\lambda_1)[-1]$ inherits   an ordinary  Leibniz algebra from $\lambda_2$. Thus, from the Leibniz$_\infty[1]$ $C(\g)$-algebra   $C(\g,M[1])$ declared by Theorem \ref{Thm:wLP to HLA} whose first structure map is $ d_{\tot}^{M[1]}$,  the   cohomology $H_{\tot}(\g,M)$   should be endowed with a   Leibniz    bracket. However, this bracket is simply zero --- see \cite{Qiaoold}*{Proposition 2.9 and Remark 2.12}. Thereby, it is futile to think about $H_{\tot}(\g,M)$. What should really be examined is the tangent cohomology $\tanH^\bullet(M)=H(M,d^M_0)$.

\begin{prop}\label{main prop}
  Let $(f\colon M \rightsquigarrow \g) \in \weakLPmodg$ be a weak Loday-Pirashvili module.
   The induced map of tangent cohomology spaces $\tanH(f_0) \colon \tanH^\bullet(M) \to \tanH^\bullet(\g)$ is a Loday-Pirashvili module. Moreover, the associated graded Leibniz algebra structure $\diamond$ on $\tanH^\bullet(M)$ (see Remark \ref{Rem:LPgradedcase}) can be expressed by truncation of the Leibniz$_\infty[1]$ $C(\g)$-algebra  structure on $C(\g,M[1])$ as follows --- For any $ u^{\sim} , v^{\sim} \in \tanH^\bullet(M)$ (where $u,v\in M$ are $d^M_0$-closed, $u^{\sim}$ and $v^{\sim}$ denote the corresponding cohomology classes), one has
   \begin{equation*}\label{circandPrM}
    u^{\sim} \diamond v^{\sim}= (-1)^{\abs{u}-1} \big(\mathrm{Pr}_{M}\circ \lambda_2(u[1],v[1])[-1]\big)^{\sim},
   \end{equation*}
   where $\mathrm{Pr}_M:~C(\g,M)\to M$ is the standard projection and $\lambda_2$ is the binary bracket of the Leibniz$_\infty[1]$ $C(\g)$-algebra    $C(\g,M[1])$.
\end{prop}

The results  \cite{Qiaoold}*{Theorem 1.20 and Proposition 3.9} are also relevant to our discussion above, and interested readers are advised to take a closer look.

\subsection{Functoriality}\label{Sec:Functoriality}
We now study a functorial property of our construction of Leibniz$_\infty$ algebras from weak {\LP modules} when the target dg Lie algebra varies.
\begin{Def}\label{Def:morphismofLPmodules}
A morphism $\Phi$ of weak {\LP modules}  from $ f \colon M \rightsquigarrow \g \in \weakLPmodg $   to $ f^\prime \colon M^\prime \rightsquigarrow \g^\prime  \in   \weakLPmodgprime $  consists of a pair $(\phi,\underline{\phi})$, where
\begin{compactenum}
  \item $\underline{\phi} \colon \g \to \g^\prime$ is a morphism of dg Lie algebras, which pulls back the   $\infty$-morphism $ f^\prime\colon M^\prime\rightsquigarrow \g^\prime  \in   \lfMod_{\g^\prime}^\infty (M^\prime,\g^\prime)$ to {an $\infty$-morphism}   $ \underline{\phi}^{\vee}f^\prime\colon M^\prime \rightsquigarrow \g^\prime  $ $\in$ $ \lfMod_\g^\infty (M^\prime,\g^\prime)$ in an obvious manner. Here we treat $M^\prime$ and $\g^\prime$ as {locally finite} $\infty$-$\g$-modules via $\underline{\phi}$.
  \item $\phi \colon M \rightsquigarrow M^\prime$ is {an $\infty$-morphism} of {\LF} $\infty$-$\g$-modules such that the following diagram commutes:
      \[
    \begin{tikzcd}
                   M  \arrow[r,"f"] \arrow[d,  rightsquigarrow, "\phi"] & \g \arrow[d, "\underline{\phi}"] \\
                   M^\prime \arrow[r,  rightsquigarrow, "\underline{\phi}^{\vee} f^\prime"] & \g^\prime,
    \end{tikzcd}
   \]
  i.e.,
  \[
           \underline{\phi} \circ f  =  \underline{\phi}^{\vee}f^\prime \circ \phi  \in  \lfMod_\g^{\infty}(M,\g^\prime).
  \]
\end{compactenum}
\end{Def}
It can be directly verified that the collection of weak {\LP modules}, together with morphisms as defined above, forms a category which we denote by $ \mathbf{wLP} $.

Given two weak {\LP module}s   $ f  \colon M \rightsquigarrow \g $ and $ f^\prime  \colon {M^\prime} \rightsquigarrow \g^\prime $, by Theorem~\ref{Thm:wLP to HLA} and using the functors $\mathcal{G}_\g^\infty$ and $\mathcal{G}_{\g^\prime}^\infty$, respectively, there are two Leibniz$_\infty[1]$ algebras, namely $(C(\g, M[1]),\lambda_\bullet)$ over $C(\g)$  and   $(C(\g^\prime, M^\prime[1]),\lambda^\prime_\bullet)$ over $C(\g^\prime)$.
The next proposition implies that   $\mathcal{G}_\g^\infty$ depends on $\g$  in a  functorial manner.
\begin{prop}\label{prop: functorial wrt g}
Let $\Phi = (\phi,\underline{\phi})$ be a morphism  from objects $ f $ to $ f^\prime$ in the category $ \mathbf{wLP} $.
\begin{compactenum}
  \item The Koszul dual $\underline{\phi}^{\vee} \colon C(\g^\prime) \to C(\g)$ of the morphism $\underline{\phi} \colon \g \to \g^\prime$ of dg Lie algebras pulls back the Leibniz$_\infty[1]$ $C(\g^\prime)$-algebra $(C(\g^\prime, M^\prime[1]),\lambda_\bullet^\prime)$ to the Leibniz$_\infty[1]$ $C(\g)$-algebra $(C(\g, M^\prime[1])$, $\mu_\bullet^\prime)$, where
      \[
        \mu_k^\prime(m_1^\prime[1],\cdots,m_k^\prime[1]) = (\underline{\phi}^{\vee} \otimes \id_{M^\prime})(\lambda_k^\prime(m_1^\prime[1], \cdots, m_k^\prime[1])),
      \]
      for all $k \geqslant 1$ and all homogeneous $m_1^\prime,\cdots,m_k^\prime \in {M^\prime}$.
  \item The image $\secondC(\g,\phi)$ of the $\infty$-morphism $\phi \colon M \rightsquigarrow {M^\prime}$ of $\g$-modules under the Chevalley-Eilenberg functor $\secondC(\g,-)$ defines a strict morphism of Leibniz$_\infty[1]$ $C(\g)$-algebras
      \[
         \secondC(\g,\phi) \colon (C(\g,M[1]),\lambda_\bullet) \to (C(\g,{M^\prime}[1]),\mu_\bullet^\prime).
      \]
\end{compactenum}
\end{prop}
\begin{proof}
The first statement is obvious. We prove the second one.
By abuse of notations, we simply write the  $C(\g)$-module morphism $\secondC(\g,\phi)$ as $\phi$.
It thus suffices to prove that
\begin{equation}\label{Eq: strict morphism in functorialy}
  \phi(\lambda_n(m_1[1], \cdots,m_n[1])) = \mu_n^\prime(\phi(m_1[1]),\cdots,\phi(m_n[1])),
\end{equation}
for all $n \geqslant 1$ and homogeneous elements $m_1,\cdots,m_n \in M$.

When $n = 1$, Equation~\eqref{Eq: strict morphism in functorialy} holds, since $\lambda_1$ and $\mu_1^\prime$ are differentials on $\secondC(\g,M)$ and $C(\g,{M^\prime})$, respectively, and $\phi \colon \secondC(\g,M) \to C(\g,{M^\prime})$ is a   $C(\g)$-module morphism by assumption.

For $n \geqslant 2$, consider $\underline{\phi}^{\vee} f^\prime \in \lfMod_\g^\infty({M^\prime}, \g^\prime)$, the pullback $\infty$-morphism of $f^\prime \in \lfMod_{\g^\prime}^\infty$ along the morphism $\underline{\phi} \colon \g \to \g^\prime$ of dg Lie algebras.
On the one hand, using Equation~\eqref{align:firstlambdan},  we have
\begin{align*}
  \phi(\lambda_{n}(m_1[1],\cdots,m_{n-1}[1],m_{n}[1])) &= (-1)^{\abs{m_1}+\cdots+\abs{m_{n-1}}-n+1} \iota_{f(m_1)[1]}\cdots \iota_{f(m_{n-1})[1]} \phi(d_{\tot}^M m_{n}[1]) \\
  &= (-1)^{\abs{m_1}+\cdots+\abs{m_{n-1}}-n+1} \iota_{\underline{\phi}(f(m_1))[1]}\cdots \iota_{\underline{\phi}(f(m_{n-1}))[1]} d_{\tot}^{M^\prime}\phi(m_{n})[1].
\end{align*}
By the second condition of Definition \ref{Def:morphismofLPmodules}, we obtain a commutative diagram in the category $\Mod_{C(\g)}$ of   $C(\g)$-modules
  \begin{equation*}\label{Eq: CD for morphisms}
   \begin{tikzcd}
     \secondC(\g,M) \arrow[rr, "\id_{C(\g)} \otimes f"]  \arrow[d,  "\id_{C(\g)}\otimes \phi"]  && \secondC(\g,\g) \arrow[d, "\id_{C(\g)}\otimes \underline{\phi}"] \\
     \secondC(\g,{M^\prime}) \ar[rr, "\id_{C(\g)} \otimes \underline{\phi}^{\vee} f^\prime"] && \secondC(\g,\g^\prime).
   \end{tikzcd}
  \end{equation*}
Thus, we can examine the right hand side of Equation~\eqref{Eq: strict morphism in functorialy}:
\begin{align*}
  \mu_n^\prime(\phi(m_1[1]),\cdots, &\phi(m_n[1])) = (\underline{\phi}^\vee \otimes \id_{M^\prime})(\lambda_n^\prime(\phi(m_1),\cdots, \phi(m_n))) \\
  &= (-1)^{\abs{m_1}+\cdots+\abs{m_{n-1}}-n+1} (\underline{\phi}^\vee \otimes \id_{M^\prime})(\iota_{f^\prime(\phi(m_1))[1]} \cdots \iota_{f^\prime(\phi(m_{n-1}))[1]}  d_{\tot}^{M^\prime}\phi(m_{n})[1]) \\
  &= (-1)^{\abs{m_1}+\cdots+\abs{m_{n-1}}-n+1}\iota_{\underline{\phi}^{\vee} f^\prime (\phi(m_1))[1]} \cdots \iota_{\underline{\phi}^{\vee} f^\prime (\phi(m_{n-1}))[1]} d_{\tot}^{M^\prime}\phi(m_{n})[1] \\
  &=  (-1)^{\abs{m_1}+\cdots+\abs{m_{n-1}}-n+1} \iota_{\underline{\phi}(f(m_1))[1]} \cdots \iota_{\underline{\phi}(f(m_{n-1}))[1]} d_{\tot}^{M^\prime} \phi(m_{n})[1].
\end{align*}
Hence, Equation~\eqref{Eq: strict morphism in functorialy} holds for all $n \geqslant 2$ as well.
\end{proof}

\begin{Rem}\label{Rem:furtherthoughts}
  As a consequence of Theorem~\ref{Thm:wLP to HLA} and Proposition~\ref{prop: functorial wrt g}, we obtain a functor $\mathcal{\G}^\infty$ from the category $\mathbf{wLP}$ of weak Loday-Pirashvili modules to the category $\mathbf{Leib}^{\infty[1]}$ of Leibniz$_\infty[1]$ algebras, if we forget about the dg module structures over dg Lie algebras. The functor $\mathcal{G}^\infty$ can be viewed as a homotopy lifting of the construction of Leibniz algebras out of \LP modules (i.e. equivariant linear maps)   in~\cite{LP}. It is natural to expect that there is a left adjoint functor $\mathcal{F}^\infty$ of $\mathcal{G}^\infty$ such that the associated adjunction is a homotopy lifting of the one~\eqref{Eq: Rep to Leib functor}. This question will be investigated in the future.
\end{Rem}

\section{Leibniz infinity algebra structures via rooted trees}\label{Sec:rootedtree}
This section is an addition to the previous one, so we inherit all the earlier symbols and conventions. Before introducing the main result, we need to recall some basic terminology about trees.
\subsection{Rooted trees with {\MO}s}
A directed graph $\Gamma$  is a pair $({\mathrm{V}(\Gamma)},{\mathrm{E}(\Gamma)})$, where ${\mathrm{V}(\Gamma)}$, called the set of vertices, is a non-empty finite set, and ${\mathrm{E}(\Gamma)}$, called the set of edges, is a set of ordered pairs of elements of ${\mathrm{V}(\Gamma)}$.
A directed tree is a connected directed graph without cycles. The valency of a vertex $v$ in a directed tree, denoted by $\doubleabs{v}$, is the  number of edges \textit{pointing to} $v$.  {Note that our definition of valency   is slightly different from that in~\cite{DR}. }

In what follows, we assume that $T$ is a  rooted tree (also known as a planted tree), i.e.,   a directed tree   whose set ${\mathrm{V}(T)}$ of vertices admits a distinguished element $v_R \in {\mathrm{V}(T)}$ of valency $1$, called the \emph{root vertex}, such that   $T$ is oriented toward   $v_R$.
The edge adjacent to $v_R$ is called the \emph{root edge}. Vertices of valency $0$ are called \emph{leaves} (or tails). A vertex is called \emph{internal} if it is neither a root nor a leaf. We use the symbols  ${\mathrm{Int}(T)}$ and ${\mathrm{Leaf}(T)}$ to denote the sets of internal vertices and    leaves, respectively.

Each non-root vertex $v$ of $T$ admits a unique edge initiating from $v$. Denote by $\mathrm{Next}(v)$ the end vertex of this edge. There is a unique path $p_v$ connecting $v$ to the root vertex $v_R$.
Call the number of edges in this path the height of $v$, and denote it by $h_v$.
For the root vertex $v_R$, set $\mathrm{Next}(v_R)=v_R$ and $h_{v_R}=0$. Meanwhile, we can define a partial order $\prec$ on the set ${\mathrm{V}(T)}$: $v \prec v'$ for any vertex $v'$ in $p_v\setminus \{v\}$.

 We define the height $h(T)$ of $T$ to be the maximum of heights of all vertices of $T$, i.e.,  $h(T) = \max\{h_v \mid v \in {\mathrm{V}(T)}\}$.
Assume that   $h(T)$  is $k+1$. We can   decompose   the set of vertices
\begin{align*}
 {\mathrm{V}(T)} &= \sqcup_{j = 0}^{k+1} {\mathrm{V}^j(T)},
\end{align*}
where ${\mathrm{V}^j(T)}$ consists of vertices of height $j$. Both ${\mathrm{V}^0(T)} = \{v_R\}$ and ${\mathrm{V}^1(T)} = \{\mathrm{Next}^{-1}(v_R)\}$ are single-element sets, and   ${\mathrm{V}^{k+1}(T)}$ is a subset of ${\mathrm{Leaf}(T)}$.
We also have a decomposition of   ${\mathrm{Int}(T)}$ of internal vertices
\[
  {\mathrm{Int}(T)} = \sqcup_{j = 1}^{k} {\mathrm{Int}^{j}(T)} := \sqcup_{j = 1}^{k} {\mathrm{V}^j(T)} \cap {\mathrm{Int}(T)}.
\]

Let $\RT(n)$ be the set of rooted trees with $n$ non-root vertices. Those rooted trees with a particular type of labelling  on vertices are what we concern.
\begin{Def}
A \textbf{\MO}  on a rooted tree $T \in \RT(n)$ is  a  bijection
\[
  l \colon {\mathrm{V}(T)} \setminus \{v_R\} \to  \{1,2,\cdots,n\}
\]
which preserves the partial order $\prec$,  i.e.,   $v_1 \prec v_2$ implies   $l(v_1) < l(v_2)$.
The collection of {\MO}s of a rooted tree $T$ is denoted by $\Oder(T)$.

Two {\MO}s $l$ and $l^\prime$ on $T$ are said to be equivalent, if there exists an automorphism $\sigma: {\mathrm{V}(T)} \to {\mathrm{V}(T)}$ of the set ${\mathrm{V}(T)}$ of vertices satisfying the following conditions:
\begin{compactenum}
  \item ($\mathrm{Next}$-compatibility) $\sigma(\mathrm{Next}(v)) = \mathrm{Next}(\sigma(v))$ for all $v \in {\mathrm{V}(T)}$;
 \item (Labelling transfer) $l^\prime(v) = l(\sigma(v))$ for all $v \in {\mathrm{V}(T)} \setminus \{v_R\}$.

\end{compactenum}
Denote by $[\Oder(T)]$ the equivalence classes of {\MO}s on $T$.
\end{Def}
\begin{Ex}\label{Ex:T1T2}
As an example, below is a tree $T_1\in  \RT(4) $ with two equivalent {\MO}s:
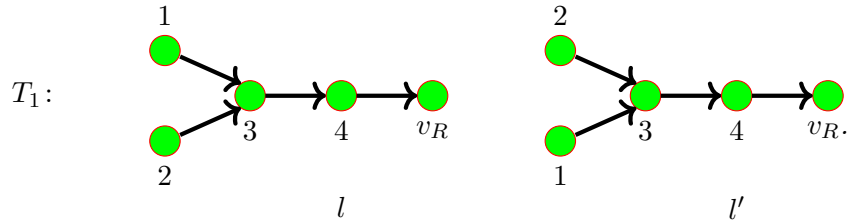
\[
  \begin{tikzpicture}[xscale=4,yscale=4]
   \draw[red,fill=green] (1.7,5) circle [radius=0.05];
  \draw[ultra thick, ->, black]  (1.47,5.13) -- (1.67,5.04);
   \draw[red,fill=green] (1.42,5.15) circle [radius=0.05];
   \draw[ultra thick, ->, black] (1.47,4.87) -- (1.67,4.96);
   \draw[red,fill=green] (1.42,4.85) circle [radius=0.05];
   \draw[ultra thick,->,black] (1.75,5) -- (1.95,5);
   \draw[red,fill=green] (2.0,5) circle [radius=0.05];
   \draw[red,fill=green] (2.3,5) circle [radius=0.05];
   \draw[ultra thick, ->, black] (2.05,5) -- (2.25,5);
   \node[below] at (2,4.7) {$l$};
   \node[above] at (1.42,5.2) {$1$};
   \node[below] at (1.42,4.8) {$2$};
   \node[below] at (1.7,4.95) {$3$};
   \node[above] at (1, 4.93) {$T_1\colon $};
   \node[below] at (2.0,4.95) {$4$};
   \node[below] at (2.3,4.95) {$v_R$};

   \draw[red,fill=green] (3.0,5) circle [radius=0.05];
  \draw[ultra thick, ->, black] (2.77,5.13) -- (2.97,5.04);
   \draw[red,fill=green] (2.72,5.15) circle [radius=0.05];
   \draw[ultra thick, ->, black] (2.77,4.87) -- (2.97,4.96);
   \draw[red,fill=green] (2.72,4.85) circle [radius=0.05];
   \draw[ultra thick,->,black] (3.05,5) -- (3.25,5);
   \draw[red,fill=green] (3.3,5) circle [radius=0.05];
   \draw[red,fill=green] (3.6,5) circle [radius=0.05];
   \draw[ultra thick, ->, black] (3.35,5) -- (3.55,5);
    \node[below] at (3.3,4.7) {$l^\prime$};
   \node[above] at (2.72,5.2) {$2$};
   \node[below] at (2.72,4.8) {$1$};
   \node[below] at (3.0,4.95) {$3$};
   \node[below] at (3.3,4.95) {$4$};
   \node[below] at (3.6,4.95) {$v_R$.};
\end{tikzpicture}
\]

Here is another tree  $T_2\in  \RT(4) $ with two   {\MO}s which are   not equivalent: 
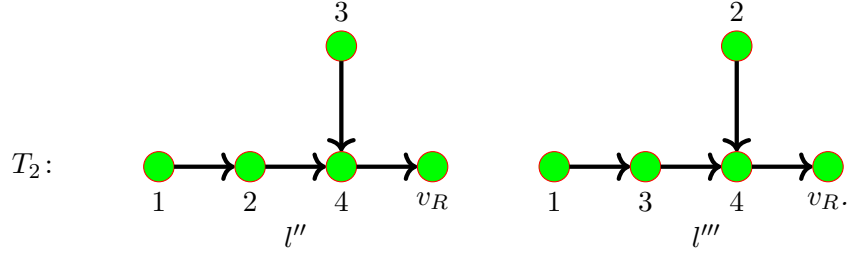
\[
\begin{tikzpicture}[xscale=4,yscale=4]
  \draw[red,fill=green] (1.4,5) circle [radius=0.05];
  \draw[ultra thick, ->, black] (1.45,5) -- (1.65,5);
  \draw[red,fill=green] (1.7,5) circle [radius=0.05];
  \draw[ultra thick,->,black] (1.75,5) -- (1.95,5);
    \draw[red,fill=green] (2.0,5.4) circle [radius=0.05];
  \draw[ultra thick, ->, black] (2.0,5.35) -- (2.0,5.05);
  \draw[red,fill=green] (2.0,5) circle [radius=0.05];
  \draw[ultra thick, ->, black] (2.05,5) -- (2.25,5);
  \draw[red,fill=green] (2.3,5) circle [radius=0.05];
   \node[above] at (1.85,4.7) {$l''$};
  \node[below] at (1.4,4.95) {$1$};
   \node[below] at (1.7,4.95) {$2$};
   \node[above] at (2.0,5.45) {$3$};
   \node[above] at (1, 4.93) {$T_2\colon $};
   \node[below] at (2.0,4.95) {$4$};
   \node[below] at (2.3,4.95) {$v_R$};

  \draw[red,fill=green] (2.7,5) circle [radius=0.05];
  \draw[ultra thick, ->, black] (2.75,5) -- (2.95,5);
  \draw[red,fill=green] (3.0,5) circle [radius=0.05];
  \draw[ultra thick, ->, black] (3.05,5) -- (3.25,5);
  \draw[red,fill=green] (3.3,5.4) circle [radius=0.05];
  \draw[ultra thick, ->, black] (3.3,5.35) -- (3.3,5.05);
  \draw[red,fill=green] (3.3,5) circle [radius=0.05];
  \draw[ultra thick,->,black] (3.35,5) -- (3.55,5);
  \draw[red,fill=green] (3.6,5) circle [radius=0.05];
  \node[above] at (3.3,5.45) {$2$};
  \node[above] at (3.2,4.7)  {$l'''$};
  \node[below] at (2.7,4.95) {$1$};
  \node[below] at (3.0,4.95) {$3$};
 \node[below] at (3.3,4.95) {$4$};
  \node[below] at (3.6,4.95) {$v_R$.};
\end{tikzpicture}
\]

\end{Ex}

\subsection{Structure maps in terms of rooted trees}

We need the generalized Nijenhuis-Richardson product\footnote{For the original Richardson-Nijenhuis product, see \cites{NR, LMS}.}
\[
\bullet_n \colon    \big(C^{q_1}(\g,\g  )\otimes\cdots \otimes C^{q_n}(\g,\g  )  \big)\times C^{p}(\g,N  ) \to C^{q_1+\cdots+q_n+p-n}(\g,N  )[-n]
\]
(for all $p \geqslant n \geqslant 1$) which is  defined   explicitly by
\[
\big((\omega_1\otimes x_1 ),\ldots,(\omega_n\otimes x_n )\big)\bullet_n  (\alpha\otimes w )
:=\pm (\omega_1\odot\cdots\odot\omega_n)\odot (\iota_{  x_1[1]}\cdots \iota_{ x_n[1]}  \alpha)  \otimes w ,
\]
where $\omega_i\in C^{q_i}(\g)$, $\alpha\in C^p(\g)$, $x_i\in \g$, $w\in N$. The sign $\pm$ is determined by the following rule --- Switching the positions of $x_i$ and $\omega_j$ produces a sign $(-1)^{(\abs{x_i}-1)\abs{\omega_j}}$.

Given a weak {\LP module} $f \colon M \rightsquigarrow \g$, we have a  $C(\g)$-module morphism $$F=\secondC(\g,f) \colon \secondC(\g,M) \to \secondC(\g,\g).$$
We now explain how to associate a multi-linear map $$\Theta_T^l \colon  M ^{\otimes n} \to C(\g,\g)[1-n]$$ with a rooted tree $T \in \RT(n)$ and  a {\MO} $l$ on $T$ --- Suppose that the height $h(T) = k+1$.
Follow the procedure below to find     $\Theta_T^l(m_1,\cdots, m_n)$ for $m_1,\cdots, m_n \in M$:
\begin{compactenum}
  \item Label each non-root vertex $v \in {\mathrm{V}(T)} \setminus \{v_R\}$ by the element $m_{l(v)}$.
  \item For every leaf $v_t \in {\mathrm{Leaf}(T)}$, replace its label  with $L(v_t) = {F}(m_{l(v_t)}) \in C(\g,\g)$.
  \item Replace labels on internal vertices of height greater than $1$ with elements in $C(\g,\g)$ inductively on the heights of vertices in a decreasing order:
  \begin{itemize}
   \item For each $v_k \in {\mathrm{Int}^{k}(T)}$, suppose  that $\mathrm{Next}^{-1}(v_k)$ is composed of vertices   $v_{k+1}^1,\cdots, v_{k+1}^{\doubleabs{v_k}}$  which are leaves of $T$. We replace the label on $v_k$ with
       \[
      L(v_k) :=  \left( L(v_{k+1}^1), \cdots, L(v_{k+1}^{\doubleabs{v_k}}) \right) \bullet_{\doubleabs{v_k}} {F}(m_{l(v_k)})  \in C(\g,\g).
       \]
         \item Assume that all internal vertices   $v_j \in {\mathrm{Int}^{j}(T)}$ for some $3 \leqslant j \leqslant k$ have been relabelled by $L(v_j) \in C(\g,\g)$. Then for an internal vertex $v_{j-1} \in {\mathrm{Int}^{j-1}(T)}$ of height $j-1$ such that $\mathrm{Next}^{-1}(v_{j-1}) = \{v_j^1,\cdots,v_j^{\doubleabs{v_{j-1}}}\} \subset {\mathrm{V}^{j}(T)}$,   relabel   $v_{j-1}$ by
       \[
          L(v_{j-1}):=  \left(L(v_j^1), \cdots, L(v_j^{\doubleabs{v_{j-1}}}) \right) \bullet_{\doubleabs{v_{j-1}}} {F}(m_{l(v_{j-1})}) \in C(\g,\g).
       \]
   \end{itemize}
\item  Finally, relabel the unique internal vertex   $v_1 \in {\mathrm{Int}^{1}(T)}$ by
    \begin{equation}\label{Eq: PhiTl}
     \Theta_T^l(m_1,\cdots, m_n) = L(v_1) :=  \left(L(v_2^1), \cdots, L(v_2^{\doubleabs{v_{1}}})\right)\bullet_{\doubleabs{v_{1}}} {F}(m_{l(v_1)})  \in C(\g,\g) ,
    \end{equation}
    where $\{v_2^1,\cdots, v_2^{\doubleabs{v_{1}}}\} = \mathrm{Next}^{-1}(v_1) \subset {\mathrm{V}^{2}(T)}$. It is not hard to see that the degree of $\Theta_T^l(m_1,\cdots, m_n)$ is $\abs{m_1}+\cdots+\abs{m_n}+1-n$.
\end{compactenum}
\begin{Ex}
Consider the tree $T_1$ with the {\MO} $l$ as in Example \ref{Ex:T1T2}. The formula \eqref{Eq: PhiTl} of the multi-linear map $ \Theta_{T_1}^{l}$ now reads:
\[
 \Theta_{T_1}^{l}(m_1,m_2,m_3,m_4) = \left((F(m_{1}), F(m_2)) \bullet_2 F(m_3)\right) \bullet_1 F(m_4).
\]
If we use the {\MO} $l^\prime$ on $T_1$, we obtain the same result.
When we use the tree $T_2$  with {\MO} $l^{\prime\prime}$, the associated multi-linear map is given by
\[
\Theta_{T_2}^{l^{\prime\prime}}(m_1,m_2,m_3,m_4) = \left(F(m_{1}) \bullet_1 F(m_2), F(m_3)\right)  \bullet_2 F(m_4);
\]
If we are working with $T_2$ and {\MO} $l^{\prime\prime\prime}$, then we have
\[
\Theta_{T_2}^{l^{\prime\prime\prime}}(m_1,m_2,m_3,m_4) = \left(F(m_{1}) \bullet_1 F(m_3), F(m_2)\right)  \bullet_2 F(m_4).
\]	
\end{Ex}
The following lemma is a direct consequence of our previous construction.
\begin{lem}\label{Lem: Invariance for equiv ordering }
  For any $T \in \RT(n)$ and {\MO} $l$ of~ ~$\,T$, the map $\Theta_T^l$: $\otimes^n M \to C(\g,\g)[1-n]$ defined in Equation~\eqref{Eq: PhiTl} only depends on the equivalence class of $l$.
\end{lem}

Let
\[
[\mathbf{MLRT}(n)] = \{(T,l) \mid T \in \RT(n), l \in [\Oder(T)]\}
\]
be the set of equivalent monotonic labelled rooted trees with $n$ non-root vertices.
We define a multi-$C(\g)$-linear map
\[
 \Theta_n\colon \otimes_{C(\g)}^{n} C(\g,M) \to C(\g,\g)[1-n]
\]
by
\begin{equation*}\label{Eq: Def of PhiT}
 \Theta_n(m_1,\cdots,m_n):= \sum_{(T,l) \in [\mathbf{MLRT}(n)]} \Theta_T^l(m_1,\cdots,m_n),
\end{equation*}
for all $m_1,\cdots, m_n \in M$.

Our main result is the following fact.
\begin{prop}\label{Prop: Leibniz infty via OPT}
  Let $f \colon M \rightsquigarrow \g$ be a weak {\LP module}. Then the higher structure maps $\{\lambda_{n+1}\}_{n \geqslant 1}$ of the Leibniz$_\infty[1]$ $C(\g)$-algebra structure on $\secondC(\g,M[1])$ are subject to  the following formula:  for all $m_1,\cdots, m_{n+1} \in M$,
  \begin{align}\nonumber
   &\quad \lambda_{n+1}(m_1[1],\cdots,m_{n+1}[1]) \\
   \label{align:secondlambdak}
   &= \sum_{k=1}^n\sum_{n_1 + \cdots n_k =n}\sum_{\sigma\in \sh(n_1,\cdots,n_k)} \frac{\epsilon(\sigma)}{k!}  \Bigl( \Theta_{n_1} , \Theta_{n_2} , \cdots,    \Theta_{n_k} \Bigr)~ \bullet_{k} ~ (d_{\tot}^{M}m_{n+1})[1] ,
  \end{align}
  where $ \Theta_{n_1}$ stands for $\Theta_{n_1}(m_{\sigma(1)},\cdots, m_{\sigma(n_1)})$, $ \Theta_{n_2}$ for $\Theta_{n_2}(m_{\sigma(n_1+1)},\cdots, m_{\sigma(n_1+n_2)})$, $\cdots$,     and $\Theta_{n_k}$ for $\Theta_{n_k}(m_{\sigma(n-n_k+1)}, \cdots, m_{\sigma(n)})$.
\end{prop}
\begin{proof}
We proceed by induction on $n$.
The Equation~\eqref{align:secondlambdak} holds obviously for the initial case $n = 1$.
Suppose that it holds for some integer $n \geqslant 1$. We consider the $n+1$ case. Using the induction assumption and Equation~\eqref{align:firstlambdan}, we have
\begin{align*}
  &\lambda_{n+2}(m_1[1],\cdots,m_{n+2}[1]) = (-1)^{\abs{m_1}-1}\iota_{f(m_1)[1]} \lambda_{n+1}(m_2[1],\cdots,m_{n+2}[1]) \\
&= \sum_{k=1}^n\sum_{n_1 + \cdots n_k =n}\sum_{\sigma\in \sh(n_1,\cdots,n_k)} \frac{\epsilon(\sigma)}{k!}  (-1)^{\abs{m_1}-1}\iota_{f(m_1)[1]} \Bigl( \Theta_{n_1},\Theta_{n_2}, \cdots, \Theta_{n_k}\Bigr) \bullet_{k} \left(d_{\tot}^{M}m_{n+2}\right)[1].
\end{align*}
Note that the contraction $\iota_{f(m_1)[1]}$ along the element $f(m_1)[1] \in C(\g,\g[1])$ is taken over all possibilities. Interpreted by the relationship with rooted trees with monotonic labelling, this contraction corresponds to either add an extra leaf labelled by $m_1$ on any one of the rooted trees defining $\Theta_{n_1}, \cdots, \Theta_{n_k}$, or add a new term $\Theta_{n_{k+1}}$ contributed by the tree $T \in \RT(1)$ with its only non-root vertex labelled by $m_1$. Hence, we obtain
 \begin{align*}
   &\quad \lambda_{n+2}(m_1[1],\cdots,m_{n+2}[1]) \\
   &= \sum_{k=1}^n\sum_{n_1 + \cdots n_k =n+1}\sum_{\sigma\in \sh(n_1,\cdots,n_k)} \frac{\epsilon(\sigma)}{k!}  \Bigl( \Theta_{n_1} , \Theta_{n_2} , \cdots,    \Theta_{n_k} \Bigr) \bullet_{k} (d_{\tot}^{M}m_{n+2})[1],
  \end{align*}
  as desired.
\end{proof}

\subsection{Examples}
\begin{Ex}
Consider the weak {\LP module} $f\colon M:= \g[1-k] \otimes \g \rightsquigarrow \g$ arising from a degree $k$ $2$-cocycle $\alpha\in Z^2(\g)$ as described in Example~\ref{Ex: pairing}.
We now write explicitly the Leibniz$_\infty[1]$ algebra structure $\{\lambda_n\}_{n \geqslant 2}$ on $C(\g,M[1])$.

The $C(\g)$-linear extension $F$ of $f$ is defined for all $x[1-k] \otimes y \in \g[1-k] \otimes \g$ by
\[
 F(x[1-k] \otimes y) = x^\sharp \otimes y \in C^1(\g, M),
\]
where $x^\sharp = \iota_{x[1]}\alpha \in \g^\vee[-1]$.
Note that the  $C(\g)$-module morphism $F$ raises the weight by $1$, i.e., $F \colon C^p(\g,M) \to C^{p+1}(\g, M)$. By the construction of the linear map $\Theta_n$ with respect to $F$, there exists only one monotonically labelled planar tree $T$ satisfying $\doubleabs{v} = 1$ for all internal vertices $v \in {\mathrm{Int}(T)}$.
Thus, we have
\[
 \Theta_n(m_1,\cdots, m_n) = \Theta_T^l(m_1,\cdots,m_n) = \prod_{i=1}^{n-1} x_{i+1}^\sharp(y_{i}[1]) x_{1}^\sharp \otimes y_{n},
\]
for all $m_i = (x_{i}[1-k], y_{i}) \in M, 1 \leqslant i \leqslant n$.
The Leibniz$_\infty[1]$ algebra structure on $C(\g, M[1])$
\[
\lambda_n\colon  \otimes^n M[1] \to C^1(\g,M[1])[1]
\]
is defined by
  \begin{align*}
    &\quad \lambda_{n+1}(m_1[1], \cdots, m_{n+1}[1]) \\
    &=\prod_{i=1}^{n-1} x_{i+1}^\sharp(y_{i}[1])x_{1}^\sharp \otimes y_{n} \triangleright (x_{n+1}[1-k], y_{n+1}) \\
    &= \prod_{i=1}^{n-1} x_{i+1}^\sharp(y_{i}[1]) x_{1}^\sharp \otimes \big((-1)^k[y_{n}, x_{n+1}][2-k] \otimes y_{n+1} + (-1)^{\abs{y_n}\abs{x_{n+1}+k-1}} x_{n+1}[1-k] \otimes [y_{n}, y_{n+1}]\big),
  \end{align*}
for all $m_i = (x_{i}[1-k], y_{i}) \in M = \g[1-k] \otimes \g, 1 \leqslant i \leqslant n+1$.
\end{Ex}

\begin{Ex}
Consider the weak {\LP module}  $f  = f_0+f_1 \colon {\bb}[-1] \rightsquigarrow \h$ arising from a Lie pair $(\largeg, \h)$ as in Example~\ref{Ex: Lie pair}. We now find the corresponding Leibniz$_\infty[1]$ algebra structure $\{\lambda_n\}_{n \geqslant 2}$ on $C(\h,\bb)$.

Note   that the $C(\g)$-linear extension $F_0$ of $f_0$ is weight-preserving, while the $C(\g)$-linear extension $F_1$ of $f_1$ raises weight by $1$.
Thus, for each $n\geqslant 2$, there exists only one monotonically labelled planar tree $T \in \RT(n)$ satisfying $\doubleabs{v} = 1$ for all internal vertices $v \in {\mathrm{Int}(T)}$.
More precisely, this map $\Theta_n$ is the $C(\g)$-linear extension of the following two linear maps
\[
\Theta^{(0)}_n \colon \bb[-1]^{\otimes n} \to \g[1-n], \qquad \Theta^{(1)}_n \colon \bb[-1]^{\otimes n} \to C^1(\g, \g)[1-n],
\]
defined respectively by
\[
 \Theta^{(0)}_n(\bar{z_1}[-1],\cdots,\bar{z_n}[-1]) = \pr_\h[\pr_\h([\pr_{\h}([\cdots \pr_\h([f_0(z_1), z_2]_{\largeg}), \cdots]_{\largeg}), z_{n-1}]_{\largeg}),z_n]_{\largeg},
\]
and
\[
	\Theta^{(1)}_n(\bar{z_1}[-1],\cdots,\bar{z_n}[-1])(x[1]) = \pr_\h[\pr_\h([\pr_{\h}([\cdots \pr_\h([\pr_\h([x,z_1]_{\largeg}), z_2]_{\largeg}), \cdots]_{\largeg}), z_{n-1}]_{\largeg}),z_n]_{\largeg},
\]
for all homogeneous $\bar{z_i}[-1] \in \bb[-1]$ and $x \in \h$.
Hence, the higher structure maps $\{\lambda_n\}_{n\geqslant 2}$ are the $C(\g)$-linear extension of the linear maps
\[
 \lambda_n^{(0)} \colon \bb^{\otimes n} \to \bb[1], \qquad  \lambda_n^{(1)} \colon \bb^{\otimes n} \to C^1(\g,\bb)[1],
\]
defined respectively by
\[
 \lambda^{(0)}_n(\bar{z_1},\cdots,\bar{z_n}) =  \pr_{{\largeg}/\h}([\pr_\h([\pr_{\h}([\cdots \pr_\h([f_0(\bar{z_1}), z_2]_{\largeg}), \cdots]_{\largeg}), z_{n-1}]_{\largeg}),z_n]_{\largeg}),
\]
and
\[
	\lambda^{(1)}_n(\bar{z_1},\cdots,\bar{z_n})(x) = \pr_{{\largeg}/\h}([\pr_\h([\pr_{\h}([\cdots \pr_\h([\pr_\h([x,z_1]_{\largeg}), z_2]_{\largeg}), \cdots]_{\largeg}), z_{n-1}]_{\largeg}),z_n]_{\largeg}).
\]
In particular, when $(\largeg, \h)$ is a pair of Lie algebras, the component $\lambda^{(0)}_n$ vanishes and the Leibniz$_\infty[1]$ algebra structure on $C(\g,\bb)$ coincides with the Leibniz$_\infty[1]$ algebra obtained from Chen, Sti\'{e}non and Xu's construction~\cite{CSX} for the Lie algebroid pair $({\largeg},\h)$ over the one-point space.
\end{Ex}

\end{document}